\documentclass[a4paper,psamsfonts,reqno]{amsart}
\usepackage{amsmath}
\usepackage{amssymb}
\usepackage{tikz-cd}
\usepackage{color}
\usepackage{verbatim}

\usepackage{amscd}
\usepackage{hyperref}

\definecolor{Blue}{rgb}{0.3,0.3,0.9}
\definecolor{Red}{rgb}{0.9,0.3,0.3}

\newtheorem{Lemma}{Lemma}[section]
\newtheorem{Th}[Lemma]{Theorem}

\newtheorem{Prop}[Lemma]{Proposition}

\newtheorem{Cor}[Lemma]{Corollary}

\theoremstyle{definition}

\newtheorem{Def}[Lemma]{Definition}
\newtheorem{Ex}[Lemma]{Example}
\newtheorem{Exs}[Lemma]{Examples}

\newtheorem*{MTh}{Main Theorem}

\theoremstyle{remark}
\newtheorem{Rem}[Lemma]{Remark}

\newtheorem{Remark}[Lemma]{Remark}
\newenvironment{Proof}{{\sc Proof.}\ }{~\rule{1ex}{1ex}\vspace{0.5truecm}}

\newcommand{\add}{\mbox{\rm add}}
\newcommand{\Add}{\mbox{\rm Add}}

\newcommand{\Tr}{\mbox{\rm Tr}}

\newcommand{\N}{\mathbb N}

\newcommand{\Z}{\mathbb{Z}}

\newcommand{\Ccal}{\mathcal{C}}

\DeclareMathOperator{\Hom}{Hom}
\DeclareMathOperator{\End}{End}

\title[Trace ideals and uniserial modules] {Trace ideals and  uniserial modules}

\begin{document}
\author{Dolors Herbera}

\address{Departament de Matem\`atiques 
Universitat Aut\`onoma de Barcelona, 08193 Bellaterra
(Barcelona), Spain }
\email{dolors.herbera@uab.cat}

\thanks{The first  author was partially supported by the projects MIMECO  PID2023-147110NB-I00 financed by the Spanish Government.}
 \author{Pavel P\v r\'\i hoda}
\address{Charles University, Faculty of Mathematics and Physics \\Department
of Algebra, Sokolovsk\'a~83,
18675 Praha 8, Czech Republic}
\email{prihoda@karlin.mff.cuni.cz}

\thanks{The second author was supported by the Czech Science Foundation grant GA\v CR 23-05148S}

\begin{abstract}
We thoroughly investigate the  trace ideals of projective modules over the endomorphism ring of a uniserial module. After the work of  Dubrovin and Puninski, it is known that this class of rings provides examples of trace ideals of projective right modules that are not trace ideals of projective left modules. In this paper we further investigate when this happens, giving an intrinsic description of such trace ideals and their properties.  We also use the theory associated to lifting projective modules modulo a trace ideal to give an alternative approach to Puninski's construction of a direct summand of a serial module that is not serial.
\end{abstract}
\date{\today}
\maketitle

Gena Puninski, back in the early 2000's, realized the importance of trace ideals of projective modules in order to organize the study of arbitrary projective modules over a ring. For example, over a noetherian ring, any idempotent ideal is the trace of a projective module and, in a number of examples, this readily shows  the existence of projective modules that are not  direct sums of finitely generated ones.

Following this track in 2010, P\v r\'\i hoda \cite{P2} started to develop the theory of what we now call relatively big projective modules and that the authors of this paper have been developing in \cite{wiegand}, \cite{AHP} and  \cite{AHP2} to study  direct sum decompositions of certain classes of modules.

A key question in this context is whether the property of being the trace ideal of a projective module is symmetric. That is, whether $I$ being the trace of a projective right $R$-module implies it is also the trace ideal of a projective left $R$-module. As far as we know, the first negative answer to this question was given by Dubrovin and Puninski in \cite{DP} and  by Dubrovin, Puninski and P\v r\'\i hoda in \cite{DPP}. 

In \cite{DP}, the authors determine all projective right and left modules, up to isomorphism, over the endomorphism ring of a certain class of uniserial modules over a particular class of chain domains called nearly simple chain domains. Computing the traces of such projective modules one deduces that there are trace ideals of projective right (left) modules that are not trace ideals  of left (right) projective modules. An essential ingredient to prove these results is P\v r\'\i hoda's Theorem \cite{pavel} showing that projective modules are determined, up to isomorphism, by their factor modulo the Jacobson radical. 

The computations in \cite{DP} heavily rely on previous work by Puninski, who had shown the existence of what is called non quasi-small uniserial modules, solving a question by   Dung and Facchini \cite[p.~111]{Dung-Facchini}, see also \cite[Chapter~11, Problem~15]{libro}. 

 Continuing the work in \cite{DP}, in \cite{DPP} the authors manage to compute all projective modules over the so called Gerasimov-Sakhaev counterexample \cite{GerSak}. This example was  constructed to show the existence of   semilocal rings with  cyclic flat modules that are not projective, answering in the negative a long standing question by Lazard. Computing the traces of  projective modules over such example also shows the set of traces of  projective right modules is not the same as the set of traces of projective left modules.

The common point between both families of examples is that they are modules over a  ring $S$ such that modulo the Jacobson radical it is the product of two division rings. Equivalently, $S$ has exactly two maximal right ideals that happen to be two-sided and that are also maximal as left ideals. Moreover, $S$ has a couple of non invertible elements $f$ and $g$ such that $f+g$ is invertible,  $gf=0$ and $fg\neq 0$. As proved by Z\"oschinger in \cite{Zoschinger}, the existence of such elements is \emph{essentially} equivalent to the existence of cyclic flat modules over $S$ that are not projective.

In \cite{pavel}, P\v r\'\i hoda proved that if $U$ is a  uniserial right module over a ring $R$, and $S$ denotes its endomorphism ring, then all projective right $S$-modules are free if and only if $S$  does not have such a pair of elements.

In  this paper we study in a systematic way the trace ideals of projective modules over the endomorphism ring of a uniserial module. Extending (and perhaps even simplifying) some of the statements that hold for the Dubrovin-Puninski example to  the endomorphism ring of an arbitrary uniserial module $U$ such that the class of direct summands of copies of $U$ contains elements not isomorphic to copies of $U$.

Our main result is the following,

\begin{MTh}\label{A}  \emph{(Theorem~\ref{TL} and Theorem~\ref{TLfp})} Let $R$ be a ring, and let $U$ be a  uniserial right $R$-module with endomorphism ring $S$. Consider the ideals of $S$
\[I=\{f\in S\mid f\mbox{ is not a monomorphism}\}\]
and 
\[K=\{f\in S\mid f\mbox{ is not an epimorphism}\}.\]
Assume that there exists  $g\in I\setminus K$ and $f\in K\setminus I$ such that $gf= 0$. Then
\begin{enumerate}
    \item $T=SfS\subseteq K$ is the trace of a countably generated projective and pure right ideal $P$ of $S$. Moreover, $T$ is pure as a right ideal;
    \item Any projective right $S$-module is isomorphic to the direct sum of a free $S$-module and a direct sum of copies of $P$; therefore, the only ideals of $S$ that are traces of projective right $S$ modules are $0, S$ and $T$.
    \item $L=SgS\subseteq I$  is the trace of a countably generated projective and pure left ideal $Q$ of $S$. Moreover, $L$ is pure as a left ideal;
    \item Any projective left $S$-module is isomorphic to a direct sum of a free module and a  module which is isomorphic to a direct sum copies of $Q$; therefore, the only ideals of $S$ that are traces of projective left $S$ modules are $0, S$ and $L$.
\end{enumerate}

If, in addition, $U$ is a finitely presented module over a chain domain $R$ then 
$T=Sf$ so $T$ is finitely generated as a left $S$-module, and $L=gS$ so $L$ is finitely generated as a right $S$-module.    
\end{MTh}

 A chain ring $R$ (that is, $R$ is uniserial as a right and as a left $R$-module) is said to be  nearly simple  if the only two-sided ideals of $R$ are $R$, $0$ and $0\neq J(R)=J(R)^2$.  Over nearly simple chain rings, it is easy to construct uniserial modules with the maps required in Theorem~\ref{A} (cf. Example~\ref{nearlysimple}), and the case of nearly simple chain domains is the one of the example of Dubrovin and Puninski that we consider in Corollary~\ref{cor:nearlysimple}.

 \medskip

In \cite{traces}, Herbera and P\v r\'\i hoda proved that, over any ring, projective modules can be lifted modulo the trace ideal of a projective module. We have developed  a number of interesting consequences  for  direct sum decompositions of modules over local noetherian rings  \cite{wiegand} and for torsion-free modules over domains  \cite{AHP}. Here we find a new one in the context of summands of a direct sum of uniserial modules. 

Recall that a module is serial if it is a direct sum of uniserial modules. 
Puninski \cite{Pun01} found an example of  a direct summand of a serial module which 
is not a direct sum of indecomposable modules. In particular, this shows that the class of serial 
modules may not be closed under direct summands. This construction answered \cite[Chapter~11, Problem~10]{libro}.
The arguments in \cite{Pun01} were   model-theoretical. 
Another approach based on the dimension theory of projective modules over semilocal rings
was described in \cite{pavel}. We will show in the final section of the paper that, surprisingly enough, Puninski's construction can also be understood 
as a consequence of the lifting of projective modules, modulo trace ideals of projectives.

\medskip

We have done our best to write the paper as self-contained as possible, including proofs of some results that are well known to illustrate the use of the  point of view of traces. For simplicity, we stick to the case of uniserial modules but it certainly seems that  similar results could hold for biuniform modules (see \cite[Lemma~4.2]{pavel}) or for the cyclically presented modules over local domains  studied by the Amini brothers and Facchini, see \cite{libro2}. 

In the first section of the paper we introduce general traces of modules, in the second  we recall  results on traces of projective modules. One of the results we stress is that traces of projective modules are determined by their factor modulo the Jacobson radical \cite[Corollary~2.9]{traces}. Therefore, if the ring modulo the Jacobson radical is the product of two division rings then it has, at most, $4$ ideals that are traces of projective right or left modules.  In the third section we examine questions of symmetry between traces of projective modules, showing in Corollary~\ref{tracenotin} that for the endomorphism ring of a uniserial module there is an intrinsic asymmetry for trace ideals of right (left) projective modules.  This implies  that the existence of a non-free projective right module over such rings is equivalent to having a   trace ideal of a right projective module that is   not the trace ideal of a left projective module.

In Section 4, we recall a general framework   to construct a class of non-free projective modules that appear in the Gerasimov-Sakhaev counterexample and in the Dubrovin-Puninski example. The only novelty is  the emphasis on trace ideals and their applications  over the endomorphism ring of uniserial modules in Proposition~\ref{gfzero}.  

In Section~5 we explicitly determine the nontrivial trace ideals described in our Main Theorem, and in Section~6 we specialize to the case of finitely presented uniserial modules over chain domains. These two sections are particularly interesting, as they are really introducing new ideas into this context that we think clarify the results previously known for nearly simple chain rings.

We switch topic in Section 7, which is devoted to the applications of the lifting of projective modules modulo a trace ideal to the direct sum decomposition of serial modules.

\medskip

The writing of this paper was motivated by the very recent application that Martini, Parra, Saor\'\i n and Virili \cite{MPSV} have found of the Dubrovin-Puninski's example to construct Grothendieck categories with some pathological properties --- we recall briefly the context they are using in Example~\ref{TTF}. We are grateful for the interesting discussions with the authors of \cite{MPSV} that motivated us to clarify the properties of the ideals that are traces of projective right or left modules in such example. The  original proofs are scattered in a number of papers, and also, at some point, the original arguments needed some countability condition on the ring that we show it is not really necessary.

We also thank Pere Ara for pointing out the existence of the Morita context in the setting of Remark~\ref{rem:tensortrace}. 

Last but not least, we thank the anonymous reviewer for his/her useful comments and suggestions.

\section{Trace ideals}

\begin{Def}
Let $R$ be a ring, and let $M$ be a right $R$-module. The \textit{trace} of $M$ in $R$ is  defined as 
    $$\Tr_R(M)=\sum _{f\in \Hom_R(M,R)} f(M).$$    
\end{Def}

\begin{Remark}\label{rem:tensortrace}
Let $R$ be a ring and let $M$ be a right $R$-module with endomorphism ring $S=\End_R(M)$. Then $M$ is also a left $S$-module. Since $\Hom_R(M,R)$ is both a right $S$-module and a left $R$-module, the tensor product $\Hom_R(M,R)\otimes_S M$ is well-defined. Then $\Tr_R(M)$ is the image of the \textit{trace map}:
    \[\begin{tikzcd}[row sep=0ex]
        \tau\colon & [-3em] \Hom_R(M,R)\otimes_S M\rar & R \\
        & f\otimes x\rar[mapsto] & f(x).
    \end{tikzcd}\]
Since $\Hom_R(M,R)$ is a left $R$-module and $M$ is a right $R$-module it follows that the image of $\tau$ is a $R$-$R$-module, that is $\Tr_R(M)$ is a two-sided ideal of $R$.

There is another natural morphism
 \[\begin{tikzcd}[row sep=0ex]
        \Phi\colon & [-3em] M\otimes _R\Hom_R(M,R)\rar & S \\
        & x\otimes f\rar[mapsto] & x\cdot f.
    \end{tikzcd}\]
where $x\cdot f\colon M \to M$ is defined by $x\cdot f (m)=xf(m)$ for any $m\in M$. Note that $x\cdot f$ is an endomorphism of $M$ that factors through $R$. In fact, the image of $\Phi$ is the two sided ideal of $S$ formed by the endomorphisms of $M$ that factor through a finitely generated free $R$-module. We usually denote the image of the pairing $\Phi$ simply by $M\Hom_R(M,R)$ see, for example, Remark~\ref{mplusM}. 

The data $(R,S, M, \Hom_R(M,R) ,\tau, \Phi) $ forms a \emph{Morita context} see, for example, \cite[\S 1.1.6]{MR}. The image of $\tau$ and the image of $\Phi$ are also known as the trace ideals of the context, see \cite[Definition~1.1]{muller}. 

The ring of this Morita context is 
$$\Lambda:=\begin{pmatrix} R&\Hom _R (M,R)\\ M&S \end{pmatrix}\cong \mathrm{End}_R(R\oplus M)$$
Let $P=\left(\begin{smallmatrix}1&0\\ 0&0\end{smallmatrix}\right)\Lambda$. We will see in Remark~\ref{mplusM} that $S/M\Hom_R(M,R)\cong \Lambda /\mathrm{Tr}_\Lambda (P)$.
\end{Remark}

The following properties of traces of modules are well known and easy to prove.

\begin{Prop} \label{traceprop} Let $R$ be a ring, and let $M$, $N$ be right $R$-modules then:
\begin{enumerate}
    \item[(i)] for any family $\{M_i\}_{i\in A}$ of right $R$-modules $\Tr_R(\oplus _{i\in A} M_i)= \sum _{i\in A}\Tr_R(M_i)$.
    \item[(ii)] For any $n\ge 1$, recall that $M^n$ is a right module over $M_n(R)$. Then $\Tr_{M_n(R)}(M^n)=M_n(\Tr_R(M))$.
    \item[(iii)] Let $J$ be a two-sided  ideal of $R$. Then $(\Tr _R (M)+J)/J \subseteq \Tr _{R /J}(M/MJ)$. If $M_R$ is projective, then $(\Tr _R (M)+J)/J = \Tr _{R /J}(M/MJ)$.
\end{enumerate}
\end{Prop}

\begin{Rem} Statement $(ii)$ of Proposition~\ref{traceprop} is a particular case of the behavior of traces under Morita Equivalence. 

Let $R$ be a ring, and let $P_R$ be a finitely generated projective module that is also a  generator. Let $S=\mathrm{End}_R(P)$. Then Morita's Theorem implies that the functor $\mathrm{Hom}_R (P, -)$ induces an equivalence between the categories of right $R$-modules and the category of right $S$-modules. 

In particular, for  any pair of
right $R$-modules $M$, $N$, there is an  isomorphism of abelian
groups
\begin{equation}\label{iso} \mathrm{Hom}_R(M,N)\to
\mathrm{Hom}_S(\mathrm{Hom}_R(P,M),\mathrm{Hom}_R(P,N))\end{equation} 
 given by
$f\mapsto f_*$ for any $f\in \mathrm{Hom} _R(M,N)$, where
$f_*(g)=f\circ g$ for any $g\in \mathrm{Hom}_R(P,M)$. Taking $N=P$, we deduce that $\mathrm{Hom}_R(M,P)\cong \mathrm{Hom}_S(\mathrm{Hom}_R(P,M),S)$ and that $\Tr _S(\mathrm{Hom}_R(P,M))$ are the $f\in S$ that factor through $M^\ell$ for some $\ell \in \N$.

If $P=R^n$ then $S=M_n(R)$. The morphisms of right $R$-modules  $R^n \to R^n$ that factor through $M^\ell$, for some $\ell$, are the ones given by left multiplication by matrices with entries  in $\Tr _R(M)=I$. Therefore, $M_n(I)$ is the trace of the right $M_n(R)$-module $\Hom _R (R^n, M)_{M_n(R)}\cong {M^n}_{M_n(R)}$.    
\end{Rem}

\begin{Lemma} \label{isoequal} Let $R$ be a ring, and let $I$ and $J$ be two-sided ideals of $R$. Assume that $I$ is the trace of a right $R$-module. Then
\begin{itemize}
\item[(i)] if $J$ is a homomorphic image of $I$ as right $R$-modules then $J\subseteq I$.
\item[(ii)] If $I\cong J$ as right $R$-modules, and $J$ is also the trace of a right $R$-module then $I=J$.
\end{itemize}
\end{Lemma}

\begin{Proof} $(i)$ Let $M$ be a right $R$-module such that $\mathrm{Tr}_R(M)=I$. Let $g\colon I\to J$ be an onto morphism. For any $x\in I$, there exists $m_1,\dots ,m_\ell \in M$ and $f_1,\dots ,f_\ell \in \mathrm{Hom}_R(M,R)$ such that $x=\sum _{i=1}^\ell f_i(m_i)$. Therefore, $g(x)=\sum _{i=1}^\ell g\circ f_i(m_i)\in I$.

$(ii)$ follows from $(i)$ and the symmetry of the hypothesis.
\end{Proof}

\begin{Ex} \label{idempotent} Let $R$ be a ring. Then any idempotent two-sided ideal $I$ of $R$ satisfies that $\Tr _R(I_R)=I$, so it is always a trace ideal of a right $R$-module. Therefore, idempotent ideals satisfy the conclusions of Lemma~\ref{isoequal}.

Let $J$ be a right ideal of $R$ that is pure in $R$, that is, $R/J$ is flat as a right $R$-module or, equivalently, for any $a\in J$, $a\in Ja$ (cf. \cite[Proposition~5]{azumaya}). Then, $I=RJ$ is an idempotent ideal of $R$. Proposition~\ref{chartraces} implies that, in this case, $I$ is also the trace of a projective right $R$-module.
\end{Ex}

Now we explain briefly the module theoretical context of TTF classes (Torsion, Torsion-Free classes) in which idempotent ideals have a crucial role. 

\begin{comment}

As a preparation, we first prove the following well known Lemma.

\begin{Lemma} \label{extri} Let $R$ be a ring, let $I$ be an ideal of $R$. Let $M$ be a right $R$-module. Then the following statements are equivalent,
\begin{itemize}
    \item[(1)] $\mathrm{Hom}_R(R/I,M)=0$ and $\mathrm{Ext}_R^1(R/I,M)=0$,
    \item[(2)] for any right $R/I$-module $N$, $\mathrm{Hom}_R(N,M)=0$ and $\mathrm{Ext}_R^1(N,M)=0$
\end{itemize}   
\end{Lemma}

\begin{Proof} We only need to prove that $(1)$ implies $(2)$. The assumption implies that for any set $A$, $\mathrm{Hom}_R((R/I)^{(A)},M)=0$ and also that $\mathrm{Ext}_R^1((R/I)^{(A)},M)=0$. Since ${}^\circ M$ is closed under homomorphic images we deduce that $\mathrm{Mod}-R/I\subseteq {}^\circ M$. 

Let $N$ be right $R/I$-module, therefore it fits in a short exact sequence of $R/I$-modules of the form 
\[0\to K\to (R/I)^{(A)}\to  N\to 0\]
applying the functor $\mathrm{Hom}_R(-, M)$ to it,  we obtain as part of the long exact sequence in homology the exact sequence 
\[0=\mathrm{Hom}_R(K,M)\to \mathrm{Ext}_R^1(N,M) \to \mathrm{Ext}_R^1((R/I)^{(A)},M)=0\]
which implies that    $\mathrm{Ext}_R^1(N,M)=0$. This finishes the proof. 
\end{Proof}

\end{comment}

\begin{Remark} \label{TTF} A class of  right R-modules is
called a \emph{torsion class} if it is closed under homomorphic images, direct sums, and group extensions, while it is called a \emph{torsion-free class} if it is closed under submodules, direct products, and group extensions.

For any class of modules $\mathcal{C}$, $${}^\circ \Ccal =\{X_R\mid \mathrm{Hom}_R(X,C)=0 \mbox{ for any $C\in \Ccal$}\}$$ is a torsion class and 
 $$ \Ccal ^\circ=\{Y_R\mid \mathrm{Hom}_R(C,Y)=0 \mbox{ for any $C\in \Ccal$}\}$$  
 is a torsion-free class. Dickson \cite{dickson} proved that all torsion classes and torsion-free classes can be defined this way.

 A class of right $R$-modules $\mathcal{C}$ is a TTF class if it is both a torsion class and a torsion-free class. Jans \cite{jans} proved that $\mathcal{C}$ is a TTF class if and only if is the class of right $R/I$-modules for an idempotent ideal $I$. Notice that then $${}^\circ \Ccal =\{ X_R\mid XI=X\} $$
 and 
 $${} \Ccal ^\circ =\{ Y_R\mid \mathrm{Ann}_M(I)=\{0\}\} .$$ The reader can check \cite{stenstrom} for a complete account on the topic.

The (right) Gabriel quotient category  of the $TTF$ class defined by an idempotent  ideal $I$ of $R$ is equivalent to the full subcategory $\mathcal{U}$ of $Mod$-$R$ whose objects are the right $R$-modules $X$ such that the natural morphism $X\to \Hom _R (I, X)$ is an isomorphism. Equivalently, $X\in \mathcal{U}$ if and only if $\Hom _R(R/I, X)=0=\mathrm{Ext}_R^1(R/I,X)$. The category $\mathcal{U}$ gives a particular kind of Grothendieck category, see \cite{roos}.

There are a number of papers, old and new, relating properties of the idempotent ideal $I$ with properties of the associated classes of right $R$-modules $({}^\circ \Ccal , \Ccal =\mathrm{Mod}-R/I ,\Ccal ^\circ)$, and of the quotient category $\mathcal{U}$. Between the very recent ones, we want to highlight the preprint by Martini, Parra, Saor\'\i n and Virili \cite{MPSV} in which they use these connections to construct  Grothendieck categories $\mathcal{U}$ with enough flat objects that does not have enough projective objects. For their construction they need examples of  ideals which are the trace of a projective right module that are not the trace of a projective left module.

In this setting of Gabriel localization, it is important to recall that a celebrated theorem by Gabriel and Popescu stated that any Grothendieck category is equivalent to the Gabriel quotient category of a module category over a (unital) ring $R$ by a hereditary torsion theory (that is, a torsion theory closed under submodules). 

\end{Remark}

\section{Traces of projective modules}

Let  $P$ be a  projective right module over a ring $R$. Then the Dual Basis Lemma implies that $P=P\mathrm{Tr}\, (P)$. Therefore, $\mathrm{Tr}\, (P)$ is the minimal ideal of $R$ having this property and it must be an idempotent ideal. J.~M.~Whitehead in \cite{whitehead} gave a characterization of traces of countably generated projective modules. Next results are further elaborations of his characterization and ideas.

Assume $P\cong E R^{(A)}$ for a suitable set $A$ and a suitable column-finite idempotent matrix  $E=(e_{ij})_{i,j \in A}$ with entries in $R$. It is not difficult to see that $\Tr_R(P)=\sum _{i,j}Re_{ij}R$. Since $E^2=E$, it follows that for each $j\in A$, $\sum _{i\in A}Re_{ij}= \sum _{i\in A}\Tr_R(P)e_{ij}$. That is, if $I$ is a trace ideal of a projective right $R$-module then it contains plenty of finitely generated (left) ideals $J\le I$ satisfying that $IJ=J$. This property characterizes trace ideals of projective modules.

\begin{Prop}[Propositions 2.4 and 2.6 in \cite{traces}]\label{chartraces} Let $R$ be a ring, and let $I$ be an ideal of $R$. Then the following statements are equivalent:
\begin{itemize}
\item[(i)] There exists a projective right $R$-module $P$ such that $I=\mathrm{Tr}\, (P)$.
\item[(ii)] For any finite subset $X$ of $I$ there exist a  finitely generated left ideal $J \le I$ such that $X\subseteq J$ and $IJ=J$.
\end{itemize}
Moreover, $I$ is the trace ideal of a countably generated projective right $R$-module if and only if there exists an ascending chain of finitely generated left ideals $(J_n)_{n\ge1}$ such that $J_{n+1}J_n=J_n$ and $I=\bigcup_{n\ge 1} J_nR$.
\end{Prop}

\begin{Cor}[Corollary 2.7 and Corollary 2.8 in \cite{traces}] \label{tracefg} Let $R$ be a ring. 
\begin{itemize}
    \item[(i)] Let $J$ be a finitely generated left
ideal such that $J^2=J$. Then $JR$ is the trace of a (countably
generated) projective right $R$-module.
\item[(ii)] Assume that $R/J(R)$ satisfies the ascending chain condition on two-sided ideals
(e.g. $R/J(R)$ left or right noetherian). Then an ideal $I$ is the
trace of a (countably generated) projective right $R$-module if and
only if there exists a finitely generated left ideal $J$ such that
$J^2=J$ and $I=JR$. In particular, all trace ideals of projective right $R$-modules  are traces of countably generated projective modules.
\end{itemize}
\end{Cor}

 From Proposition~\ref{chartraces} and Nakayama's Lemma it follows  that the only trace ideal of a projective right (or left) $R$-module  contained in the Jacobson radical is the zero ideal. This is a well-known result; for an alternative approach, see \cite[Proposition~17.14]{andersonfuller}. Next result shows that even a stronger statement holds.

\begin{Cor}[Corollary 2.9 in \cite{traces}] \label{modjR} Let $R$ be a ring, and let $I$ and $I'$ be ideals of $R$. Assume that $I$ is the trace of
a  projective right $R$-module. Then:
\begin{itemize}
    \item[(i)] If
$I+J(R)\subseteq I'+J(R)$ then $I\subseteq I'$.

\item[(ii)] If $I'$ is also the trace ideal of a   projective right $R$ module  or the trace of a projective left $R$-module then $I+J(R)=I'+J(R)$ if and
only if $I=I'$.
\end{itemize}
\end{Cor}

We will denote by $\mathcal T_r(R)$ (resp. $\mathcal T_\ell(R)$) the set of ideals of the ring $R$ that are traces of countably generated projective right (resp. left) $R$-modules. When we simply write $\mathcal T(R)$ we mean $\mathcal T_r(R)$.

By Proposition~\ref{traceprop}, $\mathcal T _r(R)$ is an additive monoid with a partial  order relation given by the inclusion. Moreover, for any two-sided ideal $J$ of $R$, the assignment $I\mapsto I+J/J$ defines a morphism of monoids $\varphi ^J_r\colon \mathcal{T}_r (R)\to \mathcal{T}_r (R/J)$. In next proposition, we summarize some of the properties of these morphisms.

\begin{Prop}[Theorem~3.1 and Corollary~2.9 in \cite{traces}] \label{mapstr} Let $R$ be a ring. With the notation above
\begin{itemize}
    \item[(i)]  If $J\in \mathcal{T}_r(R)$ then $\varphi ^J_r\colon \mathcal{T}_r (R)\to \mathcal{T}_r (R/J)$ is an onto morphism of ordered monoids.
    \item[(ii)] If $J=J(R)$ then $\varphi ^{J(R)}_r\colon \mathcal{T}_r (R)\to \mathcal{T}_r (R/J(R))$ is an injective embedding of ordered monoids, that is, $I\le I'$ if and only if $I+J(R)/J(R)\le I'+J(R)/J(R)$
\end{itemize}
\end{Prop}

Note that, by Corollary~\ref{tracefg}, if $R/J(R)$ is right or left noetherian then $\mathcal T_r(R)$ coincides with the set of ideals of $R$ that are traces of arbitrary projective right $R$-modules. As our examples are going to be over semilocal rings, that is, $R/J(R)$ is semisimple artinian,  then $\mathcal{T}_r(R)$ is the set of traces of arbitrary projective right $R$-modules and, in view of Proposition~\ref{mapstr}, $\mathcal{T}_r(R)$ is finite. In Examples~\ref{tracesd1d2}, we  explain a useful way to describe $\mathcal{T}_r(R)$ for semilocal rings.

\section{Symmetries and asymmetries}

We are interested in describing examples of rings such that $\mathcal T _r(R)\neq \mathcal T _\ell(R)$. We recall some results that would be helpful to this aim. 

First we note that traces of finitely generated projective right  modules are also traces of finitely generated projective left $R$-modules, so possible difference between trace ideals of right/left projective modules are always related to infinitely generated projective modules.

\begin{Prop} \label{trfg} Let $R$ be a ring. Let $P$ be a finitely generated projective right $R$-module and let $I$ be its trace. Then:
\begin{enumerate}
    \item[(i)] $\mathrm{Hom}_R(P,R)$ is a finitely generated projective left $R$-module and its trace is also $I$. Therefore $I\in \mathcal T_r(R) \cap \mathcal T_\ell(R)$.
    \item[(ii)] Let $X$ be a finitely generated projective right $R/J(R)$-module such that $P/PJ(R)\oplus X \cong \left( R/J(R)\right) ^{(A)}$ for some set $A$. Then there exists a  projective right $R$-module $Q$ such that $Q/QJ(R)\cong X$. In particular, $\Tr _{R/J(R)} (X) =I+J(R)/J(R)$  for some ideal $I$ which is the trace of   a projective right module over $R$. 
\end{enumerate}
\end{Prop}

Projective modules finitely generated modulo their Jacobson radical that are not finitely generated are our main (and, essentially, only) source of examples satisfying that $\mathcal T _r(R)\neq \mathcal T _\ell(R)$. The reason for such an asymmetry is explained in the following result which is just a slight variation of Theorem~7.1 and Proposition~7.3  in \cite{FHS2}.

\begin{Prop} \label{trfgmodjr} Let $R$ be a ring, and let $P$ be a countably generated projective module that is finitely generated modulo $J(R)$. Let $X$ be an  $R/J(R)$-module such that $P/PJ(R)\oplus X\cong \left( R/J(R)\right) ^n$. Then:
\begin{itemize}
    \item[(i)] There exists a countably generated projective left $R$-module $P'$ such that $P'/J(R)P'\cong \mathrm{Hom}_{R/J(R)} (X, R/J(R))$. Therefore $\Tr _{R/J(R)} (X)+J(R)/J(R)=I+J(R)/J(R)$ where $I$ is the trace of $P'$.
    \item[(ii)] $P$ is a pure submodule of $R^n_R$ and $P'$ is a pure submodule of ${}_RR^n$.
    \item[(iii)] $P$ is finitely generated if and only if so is $P'$.
    \item[(iv)] $P$ is finitely generated if and only if $\mathrm{Tr}_R(P)$ is also the trace of a countably generated projective left $R$-module that is finitely generated modulo its radical.
\end{itemize}
\end{Prop}

In the next examples we explain some tools to describe trace ideals over semilocal rings. We specially emphasize on the case in which $R/J(R)$ is the product of two division rings.

\begin{Exs} \label{tracesd1d2} (1) Assume that  $R$ is a semisimple artinian ring such that $R= M_{n_1}(D_1)\times \cdots \times M_{n_k}(D_k)$, where $D_1,\dots ,D_k$ denote division rings. Then all two-sided ideals of $R$ are traces of right (left) projective modules. Therefore, there is a bijective correspondence $\alpha  \colon \mathcal T _r (R)\to \mathcal{P}(\{1,\dots ,k\})$ by setting $$\alpha  (I)=\{i\in \{1,\dots ,k\} \mbox{ such that }e_i=(0,\dots ,1^{i)}, \dots ,0)\in I\} $$
for any $I\in \mathcal T _r (R)$. Note that $\alpha  (R)=\{1,\dots ,k\}$ and $\alpha  (0)= \emptyset$.

Of course,  in this case, $\mathcal T _r (R)=\mathcal T _\ell (R)$ and $\alpha$  describes both at the same time.

Let $V_1,\dots ,V_k$ be a set of representatives of the isomorphism classes of simple right $R$-modules such that $\mathrm{End}_R(V_i)\cong D_i$, for $i=1,\dots ,k$. Let  $P$ be a projective right $R$-module,  then  $\mathrm{Tr}_R(P)=I$ if and only if $P\cong \oplus _{i\in \alpha  (I)} V_i^{(A_i)}$ for suitable non-empty sets $A_i$.

(2) Assume  $R/J(R)\cong M_{n_1}(D_1)\times \cdots \times M_{n_k}(D_k)$, for suitable division rings $D_1,\dots ,D_k$. In view of Corollary~\ref{modjR}, the bijection $\alpha$ from $(1)$ induces an injection
\[\beta _r \colon \mathcal T _r (R)\to \mathcal{P}(\{1,\dots ,k\})\] by setting 
 $\beta _r (I)=\alpha (I+J(R)/J(R))$ for any $I\in \mathcal T _r (R)$. Similarly,  there is an embedding  \[\beta _\ell \colon \mathcal T _\ell (R)\to \mathcal{P}(\{1,\dots ,k\}).\] 
 
Denote by $\mathcal{T}_{r\ell} (R)=\mathcal T _r (R)\cup \mathcal T _\ell (R)$, that is the  ideals of $R$ that are traces of projective right modules or of projective left $R$-modules. By \ref{modjR} (ii), $\beta \colon \mathcal{T}_{r\ell} (R) \to  \mathcal{P}(\{1,\dots ,k\})$ defined by $\beta  (I)=\alpha (I+J(R)/J(R))$ is also an injective map. So $R$ can have at most $2^k$ different ideals that are traces of right or of left projective $R$-modules.

(3) Assume now that $R/J(R)\cong D_1\times D_2$ with $D_1$ and $D_2$ division rings. Then the image of $\beta _r$ is a family of subsets of $\mathcal{P}(\{1,2\})$, and there are essentially three different possibilities for the image of $\beta_r$. Namely
\begin{enumerate}
    \item[(i)] $\{\{1,2\}, \emptyset\}$: that is, the only trace ideals of projective modules are the trivial ones, which  implies that all  projective right $R$-modules are free. 
    
    This will be the case, for example, if $R$ is a commutative principal ideal domain with exactly two maximal ideals. Of course for commutative  rings, $\beta_r=\beta_\ell$. 
    \item[(ii)] $\mathcal{P}(\{1,2\})$. This happens, for example, if $R$ is a product of two local rings. For such products, of course,   the image of $\beta_r$ and of   $\beta_\ell$ coincide.
    
    \item[(iii)] $\{\{1,2\}, \{1\}, \emptyset\}$ or $\{\{1,2\}, \{2\}, \emptyset\}$. The first examples of this situation were given by Gerasimov and Sakhaev \cite{GerSak}. Though the tools to really prove that these were the only traces of projective right $R$-modules were developed much later and the real computation was completed in the paper \cite{DPP}.
    
    It was discovered by Puninski that  this is also the case of   endomorphism rings of suitable uniserial modules, see \cite{DP} or \cite{DPP}.  We  refer to this example as the Dubrovin-Puninski example because the explicit examples for the needed situation are due to Dubrovin: the so called nearly simple chain domains, see Example~\ref{nearlysimple}.

    Notice that, in the situation above, a projective right $R$-module $P$ with trace ideal $I$ such that $\beta _r(I)=\{1\}$ cannot be finitely generated. For $i=1,2$, let $V_i$ be a simple right $R$-module such that $\mathrm{End}_R(V_i)\cong D_i$. If $P$ is finitely generated, then $P/PJ(R)\cong V_1^n$ for suitable $n$. Then $V_1^n\oplus V_2^n\cong \left(R/J(R)\right) ^n$. But then, by Proposition~\ref{tracefg}, also $\{2\}$ is in the image of $\beta_r$ which is not the case. 

    In the examples given by Gerasimov and Sakhaev and also in the ones given by Puninski, the module $P$ is finitely generated modulo its Jacobson radical.  Moreover, the arguments in \cite{DP} and in \cite{DPP}, also show that the image of $\beta _r$ is $\{\{1,2\}, \{1\}, \emptyset\}$ and the image of $\beta _\ell$ is $\{\{1,2\}, \{2\}, \emptyset\}$. In this paper we are going to explain this situation in detail.
\end{enumerate}

(4)  If $R$ is two-sided noetherian then  $\mathcal{T} _r (R)= \mathcal{T} _\ell (R)$ by Proposition~\ref{chartraces}. In \cite[Example~3.5]{TAMS} we constructed examples of two-sided noetherian semilocal rings such that all finitely generated projective modules are free,  $R/J(R)$ is the product of two division rings, and  the image of $\beta_r$ (that should coincide with the image of $\beta_\ell$) covers all  three possibilities. That is, it is just  $\{\{1,2\}, \emptyset\}$, or $\{\{1,2\}, \{1\}, \emptyset\}$ (or $\{\{1,2\}, \{2\}, \emptyset\}$) or $\mathcal{P}(\{1,2\})$. 

Outside the noetherian case,  we do not know which  are all possible asymmetries that can appear between the images of $\beta _r$ and of $\beta _\ell$  in the rings of $(3)$.  Besides the remarks we already made in (iii),  in \cite[Example~3.6 (iv)]{TAMS} we also gave an example of a ring $R$ such that $R/J(R)$ is the product of two division rings and all projective right $R$-modules are free while this is not true on the left. The image of $\beta _r$ is $\{\{1,2\}, \emptyset\}$, while the image of $\beta _\ell$ for that example is $\{\{1,2\}, \{1\}, \emptyset\}$. Let us stress that, this   example is a suitable pullback of rings in which one of them is  either the Dubrovin-Puninski example or the Gerasimov-Sakhaev one.

We do not know whether there could be examples such that the image of $\beta _r$ has all the four elements in $\mathcal{P}(\{1,2\})$ while the image of $\beta _\ell$ has only two elements, and/or only three elements.
\end{Exs}

\subsection{Trace ideals over the endomorphism ring of a uniserial module.}

Let $R$ be a ring and let $U$ be a uniserial right $R$-module. Set $S=\mathrm{End}_R (U)$. After the work of Facchini \cite{facchiniwks} (alternatively, see \cite{libro}), we know that either $S$ is a local ring  or $S$ is a semilocal ring such that $S/J(S)$ is the product of two division rings. In the second case, this means that  $S$ has exactly two maximal right ideals that are two-sided and are also the two maximal left ideals. These two ideals can be described in a very precise way, they are 
\[I=\{f\in S\mid f\mbox{ is not a monomorphism}\}\]
and 
\[K=\{f\in S\mid f\mbox{ is not an epimorphism}\}.\]

\begin{Remark}\label{monoepi}
    Let $U$ and $V$ be uniserial modules. If   $f\colon U\to V$ is an injective morphism and   $g\colon U\to V$ is an onto morphism, then at least one of the maps $f$, $g$ or $f+g$ is an isomorphism (cf. \cite[Lemma~9.2(a)]{libro}).\end{Remark}

One has the following weak version of Nakayama's Lemma for these ideals of $S$.

\begin{Prop} \label{weaknakayama} Let $R$ be a ring. Let $U$ be a uniserial right $R$-module, and let $S=\mathrm{End}_R (U)$. We use the notation introduced above.
\begin{itemize}
\item[(i)] Let $J$ be a finitely generated left ideal of $S$ such that $IJ=J$ then $J=0$.
\item[(ii)] Let $L$ be a finitely generated right ideal of  $S$ such that $LK=L$ then $L=0$.
\end{itemize}
\end{Prop}

\begin{Proof} $(i)$ The key fact we are going to use in the proof is that, as $U$ is uniserial, the set of kernels of a finite number of non-zero elements in $I$ is a (totally ordered) finite subset of the submodules of $U$ so we can choose the minimal one.

Let   $J$ be a  finitely generated left ideal of $S$ such that $IJ=J$. Assume $J\neq 0$.
Let $f_1,\dots ,f_n$ be a set of generators of $J\le I$, all of them different from zero. We may assume that $\mathrm{Ker} (f_1)\subseteq \mathrm{Ker} (f_i)$ for any $i\in \{1,\dots ,n\}.$ By hypothesis, there exist $g_1,\dots ,g_n\in I$ such that $f_1=\sum _{i=1}^ng_if_i$. Choose $i_0$ such that $\mathrm{Ker}\, g_{i_0}f_{i_0}$ is minimal between the kernels of non-zero elements of the form $g_if_i$. Then 
$$\mathrm{Ker} (f_1)=\mathrm{Ker}\, g_{i_0}f_{i_0}=\mathrm{Ker}\, f_{i_0}$$
which implies that $\mathrm{Im}\, f_{i_0}\cap \mathrm{Ker}\, g_{i_0}=\{0\}$ which is not possible because both submodules are non-zero and $U$ is uniserial.  Therefore $J=0$.

 $(ii)$ In this proof we are going to use dual arguments to $(i)$. Now the key fact is that, as $U$ is uniserial, the set of images of a finite number of non-zero elements in $K$ is a (totally ordered) finite subset of the proper submodules of $U$ so we can choose a maximal one.

 Let $L $ be a nonzero finitely generated right ideal of $S$ such that $LK=L$. Let $f_1, \dots ,f_n$ be a finite set of non-zero generators of $L\le K$ and assume that $\mathrm{Im} \, f_1 \supseteq \mathrm{Im} \, f_i$ for $i\in \{1,\dots ,n\}$. Let $g_1,\dots ,g_n\in K$ be such that $f_1=\sum _{i=1}^nf_ig_i$. Let $i_0\in \{1,\dots ,n\}$ be such that $\mathrm{Im} \, f_{i_0}g_{i_0}$ is maximal. Then
$$\mathrm{Im} \, f_1= \mathrm{Im} \, f_{i_0}g_{i_0}= \mathrm{Im} \, f_{i_0}.$$
 The last equality implies that $U= \mathrm{Im} \, g_{i_0}+ \mathrm{Ker} \, f_{i_0}$. Since $\mathrm{Im} \, g_{i_0}\neq U$ and $U$ is uniserial, it follows that $\mathrm{Ker} \, f_{i_0} =U$. This is a contradiction with the assumption that $L$ is non-zero. \end{Proof}

 Combining Proposition~\ref{chartraces} with Proposition~\ref{weaknakayama} we obtain the following result on traces of projective modules over the endomorphism ring of a uniserial module.

\begin{Cor} \label{tracenotin} Let $R$ be a ring. Let $U$ be a uniserial right $R$-module, and let $S=\mathrm{End}_R (U)$. 
\begin{itemize}
\item[(i)] Let $J$ be the trace of a  projective right $S$-module. If $J\subseteq I$ then $J=0$.
\item[(ii)] Let $L$ be the trace of a  projective left $S$-module. If $L\subseteq K$ then $L=0$.
\end{itemize}
\end{Cor}

 \begin{Ex} \label{fg}  Let $U$ be a uniserial right $R$-module with endomorphism ring $S$. Assume that $I+K=S$ so that $S/J(S)$ is the product of two division rings. Let $D_1 =S/I$ and let $D_2=S/K$. 

 Combining Example~\ref{tracesd1d2} (3) with Corollary~\ref{tracenotin} we deduce that the only possibilities for the image of $\beta _r$ are either $\{\{1,2\}, \emptyset\}$ or $\{\{1,2\}, \{1\},\emptyset\}$ and the only possibilities for $\beta _\ell$ are either $\{\{1,2\}, \emptyset\}$ or $\{\{1,2\}, \{2\},\emptyset\}$.\end{Ex}

Next result is essentially well known, our variation is stating it in terms of traces of projective modules. In this result and throughout the paper, we will make free use of the fact that projective modules are isomorphic if and only if they are isomorphic modulo de Jacobson radical, which was proved, in this general form, by P\v r\'\i hoda in \cite{pavel}.

\begin{Prop}  Let $U$ be a uniserial right $R$-module with endomorphism ring $S$. Assume that $I+K=S$. Then,
\begin{itemize}
\item[(i)] \cite[Theorem~2.7]{Dung-Facchini2} All finitely generated right or left projective modules over $S$ are free;
    \item [(ii)]  $\mathcal{T}_r(S)=\{0,S\}$ if and only if all projective right $S$-modules are free;
    \item[(iii)] $\mathcal{T}_\ell(S)=\{0,S\}$  if and only if all projective left $S$-modules are free.
\end{itemize} 
In particular, $S$ has a right or left  projective module that is not free if and only if $\mathcal{T}_r(S)\neq \mathcal{T}_\ell(S)$.
\end{Prop}

\begin{Proof} If $P_S$ is a projective right $S$-module then $P/PJ(R)\cong P/PI\times P/PK \cong D_1^{(A_1)}\times D_2^{(A_2)}$, for suitable sets $A_1$ and $A_2$. As the isomorphism class of a projective module is determined modulo the Jacobson radical \cite{pavel}, the isomorphism class of $P$ is determined by the cardinality of the sets $A_1$ and $A_2$. $P$ is free if and only if $A_1$ and $A_2$ have the same cardinality.

Notice that $\mathrm{Tr}_S(P)\subseteq I$ if and only if $P=PI$ if and only if $A_1=\emptyset$. Similarly, $\mathrm{Tr}_S(P)\subseteq K$ if and only if $P=PK$ if and only if $A_2=\emptyset$. 

$(i)$ Let $P$ be a finitely generated projective right $S$-module. Consider $P/PJ(S) \cong D_1^n\times D_2^m$.

If $n\le m$ then the onto map $P/PJ(S)\to D_1^n\times D_2^n$ can be lifted to and onto map $P\to S^n$ so that $P\cong S^n\oplus Q$ with $Q/QJ(S)\cong D_2^{m-n}$. Hence, $\mathrm{Tr}_S(Q)\subseteq I$ which, by Corollary~\ref{tracenotin},  implies that $\mathrm{Tr}_S(Q) =0$ and, hence, $Q=0$. This shows that $P$ is free.

If $n>m$, then $P\cong S^m\oplus Q$ with $\mathrm{Tr}_S(Q)\subseteq K$. By Proposition~\ref{trfg}, $\mathrm{Hom}_S(Q,S)$ is a finitely generated projective left $S$-module whose trace is contained in $K$. By Corollary~\ref{tracenotin}, $\mathrm{Hom}_S(Q,S)=0$. Hence $Q=0$. So that, $P$ is free.

$(ii)$ Assume that $\mathcal{T}_r(S)=\{0,S\}$. 
 Let $P$ be an infinitely generated projective right $S$-module that is  countably generated we shall show it is free.    Consider $P/PJ(S) \cong D_1^{(A_1)}\times D_2^{(A_2)}$ and assume that   $A_1$ is finite
 and $|A_1| \leq |A_2|$. Then $P\cong S^{(A_1)}\oplus P'$ with $P'/P'J(S)\cong D_2^{(A_2')}$. By Corollary~\ref{tracenotin},  $A_2'=\emptyset$ which is not possible because $P$ is not finitely generated.

 If $A_2$ is finite and $|A_2| \leq |A_1|$, then $P\cong S^{(A_2)}\oplus P'$ with $P'/P'J(S)\cong D_1^{(A_1')}$. But then, $\mathrm{Tr}_S(P')\subseteq K$ and, by hypothesis, $P' = 0$.

The only remaining case is that $A_1,A_2$ are countably infinite, i.e., $|A_1| = |A_2| = \omega$. 
In this case $P \cong S^{(\omega)}$. 

It is clear that if all projective right $S$-modules are  free then $\mathcal{T}_r(S)=\{0,S\}$.

Statement $(iii)$ follows in a similar way. The final part of the statement is a consequence of the previous statements and Example~\ref{fg}.
\end{Proof}

\section{A general construction of projective ideals and of cyclic, countably presented, flat modules}

The following technical lemma  is an observation
based on the work by Z\"oschinger \cite[Satz~1.2]{Zoschinger}. See
also the paper by Puninski \cite[Section~3]{pun04} where all these ideas are also  explained in detail. 

\begin{Lemma} \label{productzero} Let $S$ be a ring, and  let $s\in S$. There exists a unit $u$ such that $s^2=us$ if
and only if there exists $t\in S$ such that $ts=0$ and $s+t$ is a
unit. In this situation, there exists also a unit $v\in S$ such that
$t^2=tv$.
\end{Lemma}

\begin{proof} Assume there exists a unit $u$ such that $s^2=us$. Then $t=u-s$ satisfies the required
properties. Conversely, if there exists $t\in S$ such that $ts=0$
and $s+t$ is a unit, then taking $u=s+t$ we obtain $us=s^2$.
Note that then also $tu=t^2$.
\end{proof}

In the situation of Lemma~\ref{productzero}, it follows that $Ss=SsSs$.  Therefore, by Corollary~\ref{tracefg}, $SsS$ is an idempotent ideal that is the trace of a countably generated projective right $S$-module. Similarly, $StS$ is an idempotent ideal which is the trace of a countably generated projective left $S$-module. Let $S\stackrel{s}{\to}S$ denote the map induced by the left multiplication by $s$, then  the direct limit $P$ of the countable direct system of right $S$-modules 
\[S\stackrel{s}{\to}S  \stackrel{s}{\to}S \to \cdots\]
is projective with trace  the ideal $SsS$, because the condition $s^2=us$ ensures that the (pure) exact sequence

\[0 \to S^{(\N)}\stackrel{1-\Phi}{\to} S^{(\N)} \to P\to 0\]
splits, where $\Phi$ denotes the homomorphism given by left multiplication by the matrix
\[\left(\begin{array}{ccccc}
    0& &\cdots& & \cdots\\
    s&0&\cdots & &\cdots\\
    \vdots &\ddots  & \ddots &  \\
0 & \cdots&  s&0 &\cdots \\
\vdots &  & & \ddots & \ddots\\
\end{array}\right).\]

Similarly, one can construct, a countably generated projective left $S$-module with trace ideal $StS$.

Next result gives a construction of a countably generated pure, hence projective right ideal $P$ of $S$ with trace $SsS$, and of a countably generated pure, hence projective, left ideal $Q$ with trace $StS$.

\begin{Prop}\label{rightleft} Let $S$ be a ring. Let $s$ and $u$ be elements of $ S$
such  that $u$ is invertible and $s^2=us$. For  every $m\in
\mathbb{Z}$, set   $s_m=u^{-m }(u^{-1}s)u^m$. Let $P=\sum _{m\in
\mathbb{Z}} s_mS$, and let $Q=\sum _{m\in \mathbb{Z}} S(1-s_m)$.
Then:

\noindent{\rm (i)} For any $n>m\in \Z$, $s_ns_m=s_m$ and $(1-s_n)(1-s_m)=1-s_n$.

\noindent{\rm (ii)} The right ideal $P$ and the left ideal $Q$ are
pure countably generated and, hence, projective, and contain $s$ and $t$ respectively. 

\noindent{\rm (iii)} $\mathrm{Tr}_S(P)=SP=SsS$ and $\mathrm{Tr}_S(Q)=QS=StS$.

\noindent{\rm (iv)} The right ideal $P$ is finitely generated if
and only if  the left ideal $Q$ is finitely generated, if and only
if  $su^{-2}s=u^{-1}s$. Equivalently, the finitely generated flat right module $S/P$ is projective if and only if the finitely generated flat left module $S/Q$ is projective if and only if $su^{-2}s=u^{-1}s$.
\end{Prop}

\begin{Proof} Almost all the statement of the proposition are essentially known, a possible reference is \cite[Proposition~5.3]{FHS2}. The statements about traces are new, but they easily follow from the construction and Example~\ref{idempotent}.\end{Proof}

Now we specialize this construction to  projective modules over the endomorphism ring of a uniserial module. It is useful to keep in mind the well known relation between direct sum decompositions of a module and projective right modules over its endomorphism ring.

\begin{Prop}[Theorem 4.7 in \cite{libro}]
\label{equivalencia}
Let $R$ be a ring. Let $M$ be a right $R$-module. Let $S=\End_R(M)$. Then the functor $\Hom_R(M,-)$ induces a category equivalence between $\add (M)$ and the category of finitely generated projective right $S$-modules. The inverse of this equivalence is given by the functor $-\otimes _S M$.

Assume, in addition, that $M_R$ is finitely generated. Then the functor $\Hom_R(M,-)$ induces a category equivalence between $\Add (M)$ and the category of projective right $S$-modules. 
\end{Prop}

\begin{Prop}\label{gfzero} Let $S$ be the endomorphism ring of a finitely generated uniserial right module $U$ over a ring $R$. Assume there exist $g\in I\setminus K$ and $f\in K\setminus I$ such that $gf= 0$.  Then 
\begin{itemize}
    \item[(i)] $I$ contains a countably generated projective left ideal $Q$ containing $g$, which is pure in $S$, such that $Q/J(S)Q\cong Q/KQ\cong D_2$;
    \item[(ii)] $K$ contains a countably generated projective right ideal $P$, containing $f$ which is pure in $S$ and such that $P/PJ(S)\cong P/PI\cong D_1$; 
    \item[(iii)] any projective right $S$-module is a direct sum of a free right module and  a module isomorphic to a direct sum of copies of $P$, while any projective left $S$-module is a direct sum of a free left module and 
    a module isomorphic to a direct sum of copies of $Q$.
    \item[(iv)] $\mathcal{T}_r(S)=\{0, S, SfS\}$ and $\mathcal{T}_\ell (S)=\{0, S, SgS\}$
\end{itemize}   
\end{Prop}

\begin{Proof} By Remark~\ref{monoepi}, $f+g$ is invertible. Therefore, we are under the hypothesis of Proposition~\ref{rightleft}, and we can construct the claimed projective right ideal $P$ and the projective left ideal $Q$ from statements $(i)$ and $(ii)$, respectively. 

Statement $(iii)$ follows from Remark~\ref{fg}, while statement $(iv)$ follows  also from Proposition~\ref{rightleft} and $(iii)$.
\end{Proof}

We recall the following result, which shows that for the endomorphism ring of a uniserial right module the situation of Proposition~\ref{gfzero} is the crucial one.
    
\begin{Th} \emph{(\cite[Theorem~4.3]{pun04},\cite[Proposition~4.6]{pavel})} \label{addU}
Let $S$ be the endomorphism ring of a  uniserial right module $U$ over a ring $R$.  Then  all projective right $S$-modules are free if  and only if for any pair $f\in K\setminus I$ and $g\in I\setminus K$ we have $gf\neq 0$.
\end{Th}

In the situation of Theorem~\ref{addU} we do not know whether all projective left $S$-modules should be also free. 
\begin{comment}
\begin{Remark} \label{summandV} Let  $U$ is a uniserial right $R$-module with endomorphism ring $S$. Assume that there is $f\in I\setminus K$ and $g\in K\setminus I$ we have $fg\neq 0$. Then $u=f+g$ is a unit of $S$ such that $uf=f^2$. Therefore the morphism $\mathrm{Id}-\Phi\colon U^{(\N)}\to U^{(\N)}$, where
\[\Phi =
\left(\begin{array}{ccccc}
    
     0& &\cdots& & \cdots\\
    f&0&\cdots & &\cdots\\
    \vdots &\ddots  & \ddots &  \\
    0 & \cdots&  f&0 &\cdots \\
\vdots &  & & \ddots & \ddots\\
\end{array}\right),\]
has a left inverse. If we consider the direct limit $V$ of the direct system
\[U\stackrel{f}{\to}S  \stackrel{f}{\to}U \to \cdots\]
which is a uniserial module, because the direct limit of uniserial modules is also uniserial. The module $V$ fits into the exact sequence
\[0\to U^{(\N)}\stackrel{1-\Phi}\to U^{(\N)}\to V\to 0 \]
which is split. Therefore, $V$ is a uniserial direct summand of  $U^{(\N)}$.
\end{Remark}

\end{comment}

\section{Determining the trace ideals}

In this section we will give an intrinsic description of the nontrivial trace ideals appearing in Proposition~\ref{gfzero}. We start with some technical results useful to compare the right or left ideal generated by an endomorphism of a module.

\begin{Lemma} \label{pbpo} Let $R$ be a ring, and let  $U$, $N$ and $M$ be right $R$-modules. 
\begin{enumerate}
    \item[(i)] Let $f\in \mathrm{Hom}_R(M,U)$ be injective. Then 
    \[f\mathrm{Hom}_R(N,M)=\{g\colon N\to U\mid \mathrm{Im}\, g \subseteq \mathrm{Im}\, f\}.\]
    In particular, if $g\colon N\to U$ is an injective morphism then $\mathrm{Im}\, g = \mathrm{Im}\, f$ if and only if there exists an invertible morphism $u\colon N\to M$ such that $g=fu$.

    \item[(ii)] Let $f\in \mathrm{Hom}_R(U,M)$ be surjective. Then 
    \[\mathrm{Hom}_R(M,N)f=\{g\colon U\to N\mid \mathrm{Ker}\, f \subseteq \mathrm{Ker}\, g\}.\]
    In particular, if $g\in U\to N$ is an onto morphism then $\mathrm{Ker}\, g = \mathrm{Ker}\, f$ if and only if  there exists an invertible morphism $u\colon M\to N$ such that $g=uf$.
\end{enumerate}   
\end{Lemma}

\begin{Proof} $(i)$. It suffices to show that if $g\colon N\to U$ is such that $\mathrm{Im}\, g \subseteq \mathrm{Im}\, f$ then $g\in f\mathrm{Hom}_R(M,U)$. For this purpose, let $\varepsilon \colon \mathrm{Im}\, f \to U$ be the inclusion, then $f=\varepsilon \circ f'$ and $g=\varepsilon \circ g'$ where $f'$ and $g'$ are just the maps $f$ and $g$ but with the (possibly) smaller codomain. Note that $f'$ is an isomorphism. Now consider the pullback diagram
\[\begin{tikzcd}[column sep=huge,row sep=huge]
        X\rar {f''}\dar{g''} & N\dar{g'} \\
        M\rar{f'} & \mathrm{Im}\, f
    \end{tikzcd}\]
 in which $f''$ is an isomorphism because $f'$ is. Therefore $g=f\circ g''\circ (f'')^{-1}\in  f\mathrm{Hom}_R(N,M)$.

 If $f$ and $g$ are injective maps with the same image, then the previous argument shows that there exist $u\in \mathrm{Hom}_R(N,M)$ and $v\in \mathrm{Hom}_R(M,N)$ such that $g=fu$ and $f=gv$. The injectivity of $f$ and $g$ implies that $uv=\mathrm{Id}_M$ and $vu=\mathrm{Id}_N$.

 Statement $(ii)$ follows by a dual argument using the push-out of the induced maps $f'\colon U/\mathrm{Ker}\, f \to M$ and $g'\colon U/\mathrm{Ker}\, f \to N$.
\end{Proof}

\begin{Lemma} \label{i+kpbpo} Let $R$ be a ring, and let  $U$ be a uniserial right $R$-module. Let $S=\mathrm{End}_R(U)$ and assume that $K+I=S$, equivalently, that $K$ and $I$ are different maximal ideals of $S$. Let $g\in J(S)$.
\begin{enumerate}
\item[(i)] Let $f\in K\setminus I$, then:
\begin{itemize}
    \item[(i.1)] If $\mathrm{Im}\, g\subseteq \mathrm{Im} \, f$ then $g\in fI$.
    \item[(i.2)] If $\mathrm{Im}\, g\supseteq \mathrm{Im} \, f$ then there exists $g'\in K\setminus I$ such that $\mathrm{Im}\, g'= \mathrm{Im} \, g$, so $g\in g'I$ and $f\in g'S$.
\end{itemize}
\item[(ii)] Let $f\in I\setminus K$, then:
\begin{itemize}
    \item[(ii.1)] If $\mathrm{Ker}\, f\subseteq \mathrm{Ker} \, g$ then $g\in Kf$.
    \item[(ii.2)] If $\mathrm{Ker}\, f\supseteq \mathrm{Ker} \, g$ then there exists $g'\in I\setminus K$ such that $\mathrm{Ker}\, g'= \mathrm{Ker} \, g$, so $g\in Kg'$ and $f\in Sg'$.
\end{itemize}
\end{enumerate}  
In particular, $J(S)=I\cap K=KI$.
\end{Lemma}
\begin{Proof} $(i.1)$. By Lemma~\ref{pbpo} (i), there exists $h\in S$ such that $g=f\circ h$. Since $g$ is not injective and $f$ is injective, $h$ is not injective which is to say that $h\in I$.

$(i.2)$. Consider the canonical factorization of $g\colon U\to U$ given by the Isomorphism Theorem,
\[\begin{tikzcd}%[row sep=huge,column sep=huge]
        U \rar{g}  \drar{\pi} & U  \\
        & U/\mathrm{Ker}\, g \uar{\widetilde {g}}
    \end{tikzcd}\]
By Lemma~\ref{pbpo}(i), there exists $h\colon U\to U/\mathrm{Ker}\, g$ such that $f= \widetilde {g}\circ h $. Since $f$ is injective so is $h$. So we have an onto morphism $\pi \colon U \to U/\mathrm{Ker}\, g$ and a monomorphism $h\colon U \to U/\mathrm{Ker}\, g$, which implies that there is an isomorphism $\alpha \colon U/\mathrm{Ker}\, g \to U$. Now $f= (\widetilde {g}\circ \alpha ^{-1})\circ (\alpha \circ h)=g'\circ h'\in g'S$. Since $ \widetilde {g}$ is injective, so is $g'= \widetilde {g}\circ \alpha ^{-1}$ and $\mathrm{Im}\, \widetilde {g}= \mathrm{Im}\, g' =\mathrm{Im} \, g \neq U. $ By $(i.1)$, $g\in g'I$.

$(ii.1)$. By Lemma~\ref{pbpo} (ii), there exists $h\in S$ such that $g=h\circ f$. Since $g$ is not onto and $f$ is onto, $h$ is not onto which is to say that $h\in K$.

$(ii.2)$. Consider the canonical factorization of $g\colon U\to U$ given by the Isomorphism Theorem,
\[\begin{tikzcd}%[row sep=huge,column sep=huge]
        U \rar{g}  \drar{\pi} & U  \\
        & U/\mathrm{Ker}\, g \uar{\widetilde {g}}
    \end{tikzcd}\]
By Lemma~\ref{pbpo}(ii), there exists $h\colon  U/\mathrm{Ker}\, g \to U$ such that $f= h\circ \pi $. Since $f$ is onto so is $h$. So we have an injective morphism $\widetilde{g} \colon U/\mathrm{Ker}\, g \to U$ and an  onto morphism $h\colon  U/\mathrm{Ker}\, g \to U$, which implies that there is an isomorphism $\beta \colon U/\mathrm{Ker}\, g \to U$. Now $f= (h\circ \beta ^{-1})\circ (\beta \circ \pi)=h'\circ \pi '\in S\pi '$. Since $ \pi$ is onto, so is $\pi '= \beta \circ \pi$ and $\mathrm{Ker}\, \pi = \mathrm{Ker }\, \pi' =\mathrm{Ker} \, g \neq \{0\}. $ By $(ii.1)$, $g\in K \pi '$. This concludes the proof of the statement.
\end{Proof}

Now we are ready to give an intrinsic description of the non-trivial trace ideals appearing in  Proposition~\ref{gfzero}. We use the following definitions first introduced in  \cite{pavelqs} and in \cite{addU}.

\begin{Def} Let $U$ be a uniserial right module over a ring $R$. Let $U_e = \sum_{g \in S \setminus K} {\rm Ker}\, g$
be the union of kernels of epimorphisms of $S$. Dually, let $U_m=\bigcap _{f\in S\setminus I} \mathrm{Im}\, f$.
\end{Def} 

By Lemma~2.2 and Lemma~2.3 in \cite{pavelqs}, if $S=I+K$ then $U_e$ and $U_m$ are invariant  submodules of $U$ such that $U_e\neq \{0\}$ and  $U_m\neq U$. If there exist $g\in I\setminus K$ and $f\in K\setminus I$ such that $gf= 0$ then $U_m\subseteq U_e$.

\begin{Lemma} \label{auxTL} Let $U$ be a  uniserial right module over a ring $R$. Let $S$ be the endomorphism ring of $U$. Assume that $I+K=S$. Then:
\begin{itemize}
    \item[(i)] $T=\{f\in K\mid f(U)\subseteq U_e\}$  is a two-sided ideal of $S$.
    \item[(ii)] $L=\{g\in I\mid  U_m \subseteq \mathrm{Ker}\, g\}$  is a two-sided ideal of $S$.
    \item[(iii)] For any monomorphism $f'\in T$ there is  an epimorphism   $g'\in L$ such that $g'f'=0$.
    \item[(iv)] For any epimorphism $g'\in L$ there is a monomorphism $f'\in T $ such that $g'f'=0$.
    \item[(v)] There exist  $g\in I\setminus K$ and $f\in K\setminus I$ such that $gf= 0$ if and only if there is a monomorphism in $T$ if and only if there is an epimorphism in $L$.  
\end{itemize}
\end{Lemma}

\begin{Proof} Since $U_e$ and $U_m$ are invariant submodules, it follows that $T$ and $L$ are   two-sided ideals of $S$.  This shows $(i)$ and $(ii)$.

$(iii)$. Assume that   $f'$ is a monomorphism in $T$. So that, $f'(U)\subseteq U_e=\sum_{g \in S \setminus K} \mathrm{Ker}\, g$. We claim that there exists $g'\in S\setminus K$ such that $f'(U)\subseteq \mathrm{Ker}\, g'$. Therefore, $g'\circ f'=0$ and $g'$ is an epimorphism in $L$.

Suppose on the contrary that no such $g'$ exists. As $U$ is uniserial, this means that, for any $g\in S \setminus K$, $\mathrm{Ker}\, g\subseteq f'(U)$. Therefore, $f'(U)= U_e=\sum_{g \in S \setminus K} \mathrm{Ker}\, g$ and, hence, $f'\colon U\to U_e$ is an isomorphism. By \cite[Lemma~2.3]{pavelqs}, this implies that $U=U_e$, so that $f'$ is onto, which contradicts the hypothesis on $f'$.

$(iv)$. Let $g'$ be an epimorphism in $L$. Therefore $U_m= \bigcap _{f\in S\setminus I} \mathrm{Im}\, f\subseteq \mathrm{Ker}\, g'$. We shall prove that there exists $f'\in S\setminus I$ such that $f'(U)\subseteq  \mathrm{Ker}\, g'$, because then $g'\circ f' =0$ and $f'$ is the claimed monomorphism. 

Suppose on the contrary that such $f'$ does not exist. Since $U$ is uniserial, this implies that $\mathrm{Ker}\, g' \subseteq f(U)$ for any $f\in S\setminus I$. Since $g'\in T$ we have $U_m=\mathrm{Ker}\, g'$, and then, by the isomorphism theorem, $g'$ induces an isomorphism $U/U_m\to U$. By \cite[Lemma~2.2]{pavelqs}, $U_m=0$ and then $g'$ is injective, which is not possible. This finishes the proof of the claim.

$(iv)$. If there exist $g\in I\setminus K$ and $f\in K\setminus I$ such that $gf= 0$ then $f$ is a monomorphism with image contained in $U_e$ therefore, $f\in T$.  Moreover, $g$ is an onto map such that $U_m\subseteq \mathrm{Ker}\, g$ so $g\in L$. \end{Proof}

\begin{Th} \label{TL} Let $U$ be a  uniserial right module over a ring $R$. Let $S$ be the endomorphism ring  of $R$, assume that there exist $g\in I\setminus K$ and $f\in K\setminus I$ such that $gf= 0$. Then:
\begin{enumerate}
    \item $T=SfS=Sf'S$ for any monomorphism $f'\in T$, and $T$ is the trace of a countably generated projective and pure right ideal of $S$; 
    \item $L=SgS=Sg'S$ for any  $g'\in L$ epimorphism,  and $L$ is the trace of a countably generated projective and pure left ideal of $S$;
    \item $T$ is pure as a right ideal.
    \item $L$ is pure as a left ideal.
\end{enumerate}   
\end{Th}

\begin{Proof} $(1)$. By Proposition~\ref{gfzero}, $SfS$ is the only nontrivial element of $\mathcal{T}_r(S)$, and the trace of a countably generated and pure right ideal of $S$. By Lemma~\ref{auxTL}, $SfS\subseteq T$. 

Let $f'\in T$. If $f'$ be a monomorphism then, by  Lemma~\ref{auxTL}, there exists $g'\in I\setminus K$ such that $g'f'=0$. Therefore, by Proposition~\ref{gfzero}, $Sf'S\in \mathcal{T}_r(S)$ and then  $SfS=Sf'S$ so $f'\in SfS$.

Assume now that $f'\in T$ is not a monomorphism, so that $f'\in J(S)$. Then   $f' = f + (f'-f)$ and $f$ and $f'-f$ are monomorphisms in $T$. By the first part of the argument, they are elements of $SfS$, so $T\subseteq SfS$. This finishes the proof of $T=SfS$.

$(2)$. It follows in a symmetric way using again Proposition~\ref{gfzero} and  Lemma~\ref{auxTL}.

$(3)$ To prove that $T$ is pure as a right ideal it suffices to show that for any $h\in T$, $h\in Th$, cf. Example~\ref{idempotent}. What we are going to show is that for any $h\in T$, there is pure right ideal $J\leq T$ such that $h\in J$, so that $h\in Jh\subseteq Th$.

If $h$ in $T$ is a monomorphism, then by  Lemma~\ref{auxTL}, there exists   $g\in I\setminus K$ such that $gh=0$. So that $u=g+h$ is invertible. By Proposition~\ref{rightleft}, there exists a countably generated pure right ideal $J=\sum _{m\in \Z} u^{-m-1}hS$. Note that $h\in J\subseteq T$. 

Assume now that $h$ is not a monomorphism. By Lemma~\ref{i+kpbpo}, if $h(U)\subseteq f(U)$ then $h\in fI\subseteq fS$ which we already know it is contained in a pure right ideal inside $T$. 

If $f(U) \subseteq h(U)$ then, by Lemma~\ref{i+kpbpo}, there exists $h' \in K\setminus I$ such that $h'(U)=h(U)$, so that $h'\in T$, and $f, h\in h'S$. By, the first part of the argument $h'S$ is contained in a pure right ideal inside $T$, so also does $h$. This finishes the proof of the purity of $T$ as right ideal.

$(4)$ It is proved in a similar way of $(3)$ using  Lemma~\ref{auxTL} and Lemma~\ref{i+kpbpo}.
\end{Proof}

\begin{Ex} \label{computationueum} Let $R$ be  a right chain ring. Let $0\neq H \le J(R)$ be a   right ideal of $R$. By \cite[Lemma~4.3]{pavelqs}, $(R/H)_e=RH/H$. 

Let $\mathcal{M}(H)=\{r\in R\mid rH = H\}$ denote   the idealizer of $H$ in $R$, and set $H'=\bigcap _{r\in \mathcal{M}(H)}rR$.  By \cite[Lemma~4.4]{pavelqs}, $(R/H)_m=H'/H$.    
\end{Ex}

\section{Finitely presented uniserial modules over chain domains}

In this section we  study trace ideals over endomorphism rings of finitely presented uniserial modules over chain domains. 
General results on endomorphism rings of cyclically presented modules over local rings can be found in \cite{AAF}. We just 
briefly recall some of them in our particular setting.

Let  $R$ be a local domain which is a subring of a division ring $D$.
Fix $0 \neq r \in J(R)$ and consider a module $U = R/rR$. 
Thus $U$ has the presentation
\[0\to R\stackrel{r}{\to}R\to U\to 0.\]
If $f\colon U \to U$ is a module homomorphism, then there is a commutative diagram of module homomorphisms with exact rows

\begin{equation} \label{endU}\begin{tikzcd}%[cramped, sep=small]
	0\arrow[r]&R \arrow{r}{r} \arrow{d}{f_2} &  R  \arrow[r] \arrow{d}{f_1} &  U  \arrow[r]\arrow{d}{f}  &0\\
	0 \arrow[r] &R \arrow{r}{r} &  R \arrow[r]  &  U\arrow[r]    & 0
\end{tikzcd}
\end{equation}
where $f_1$ and $f_2$ are uniquely determined up to homotopy. Since $\Hom _R(U, R)=0$, $f_1$ is homotopic to zero if and only its image is in $rR$.

\subsection{The idealizer of \texorpdfstring{$rR$}.} \label{idealizer}

Let $S'=\{t \in R \mid tr \in rR\}$ be the idealizer of $rR$ in $R$. In view of diagram \ref{endU}, it is easy to argue that there is a canonical onto ring morphism $S' \to \mathrm{End}_R(U)$ with kernel $rR$. Therefore, if the endomorphism $f\colon U\to U$ is induced by multiplication by $t\in S'$ and $tr=rt'$ the diagram \ref{endU} becomes

\begin{equation} \label{endUI}\begin{tikzcd}%[cramped, sep=small]
	0\arrow[r]&R \arrow{r}{r} \arrow{d}{t'} &  R  \arrow[r] \arrow{d}{t} &  U  \arrow[r]\arrow{d}{f}  &0\\
	0 \arrow[r] &R \arrow{r}{r} &  R \arrow[r]  &  U\arrow[r]    & 0
\end{tikzcd}
\end{equation}
If we assume that $f\neq 0$, left multiplication by $t$ is an injective endomorphism of $R$. Then the Snake Lemma yields a long exact sequence
\begin{equation} \label{snakelemma}0\to \mathrm{Ker}\, f \to R/t'R \stackrel{r}{\to} R/tR \to \mathrm{CoKer}\, f \to 0\end{equation}

Since $S'$ is included in a division ring $D$, it follows easily that $S'= R \cap rRr^{-1}$. Let $I' = R \cap rJ(R)r^{-1}$ and $K' = J(R) \cap rRr^{-1}$ and note that both are ideals of $S'$. 

If  $t \in S' \setminus (I'\cup K')$ then $t\in R^{*} \cap rR^{*}r^{-1}$, so $t^{-1} \in S'$. Therefore, we can conclude that any (left, right, 2-sided) ideal is contained in $I' \cup K'$ so in fact 
is contained either in $I'$ or in $K'$. Therefore we have proved that 

\begin{Prop} Following the notation above, there are two possibilities 
\begin{enumerate}
\item[(i)] $I' \subseteq K'$ or $K' \subseteq I'$. In this case $S'$ is local and $J(S') = I' \cup K'$
\item[(ii)] $I'+K'=S'$ so that $I'$, $K'$ are the only 
maximal (left, right or 2-sided) ideals and $I'\cap K' = J(S')$. Hence, $rR \subseteq J(S')$.
\end{enumerate}
\end{Prop}

\begin{Rem} \label{ARtranspose} With the notation above, the Auslander-Reiten transpose of $R/rR$ is the module $R/Rr$. By the symmetry of the hypothesis, $\mathrm{End}_R(R/Rr)\cong S''/Rr$ where $S''= r^{-1}Rr \cap R$. Moreover the inner automorphism $\varphi \colon D \to D$ given by $\varphi (d)=r^{-1}dr$ for any $d\in D$, induces an isomorphism $S'\to S''$.
\end{Rem}

Now we want to consider the case of chain domains. Our arguments are mostly a particular case of  \S 1.2 in \cite{mathiak}. 

Recall that a ring $R$ is a chain ring if the lattice of right ideals  and the lattice of left ideals of $R$ are both totally ordered.

\begin{Lemma} Let $R$ be a domain contained in a division ring $D$. The following are equivalent,
\begin{itemize} \label{valuation}
    \item[(i)] for every $0 \neq d \in D$,
either $d \in R$ or $d^{-1} \in R$,
\item[(ii)] $R$ is a chain domain and its Ore division ring of fractions is $D$.
\end{itemize}
\end{Lemma}

In the situation of Lemma~\ref{valuation},  $R$ is said to be a \emph{valuation ring} in $D$. In this case, for any $0\neq r\in R$, $rRr^{-1}$ is also a valuation ring in $D$.

\begin{Lemma} \label{s'principal}
Assume $R$ is a valuation domain in a division ring $D$, and fix $0 \neq r \in J(R)$.
Set $S' = R \cap rRr^{-1}$. Let $H$ be a left $S'$-submodule of $R$ or of $rRr^{-1}$. The for any
$h_1,\dots,h_k \in H$ then there exists $h \in H$ such that $h_1,\dots,h_k \in S'h$.
A Similar statement holds on the right. 
\end{Lemma}

\begin{Proof} It suffices to show the  case $k=2$.
Let $H$ be a left $S'$-submodule of $R$, $h_1,h_2 \in H$. We may suppose that $h_1 \in Rh_2$.
If also $h_1 \in rRr^{-1}h_2$ then $h_1 \in S'h_2$. Assume that $h_2 \in rJ(R)r^{-1}h_1$.

If $h_1 \in R^{*}h_2$, we get $h_2\in Rh_1$ and $h_2 \in S'h_1$ follows. 
It remains to consider $h_1 \in J(R)h_2$. In this case, $h_1,h_2 \in R(h_1+h_2)$
and also $h_1,h_2 \in rRr^{-1}(h_1+h_2)$. As before, we get $h_1,h_2 \in S'(h_1+h_2)$.   
\end{Proof}

\begin{Cor}
$S'\subseteq R$  and $S' \subseteq rRr^{-1}$ are flat extensions.
\end{Cor}

\begin{Proof} It follows from Lemma~\ref{s'principal} that  $R$ and $rRr^{-1}$ are the directed union of its cyclic submodules on the right and on the left. Therefore, they are flat as $S'$-modules on either side.
\end{Proof}

We finally state the crucial property of $S'$ we are interested in: 

\begin{Lemma}\label{equalityideals} Assume $R$ is a valuation domain in a division ring $D$. Fix $0 \neq r \in J(R)$.
Set $S' = R \cap rRr^{-1}$. Then,

\begin{itemize}
    \item[(i)] If $I$ is a left ideal of $S'$. Then $I = RI \cap rRr^{-1}I$. In particular, if $I$ and $J$
are left ideals of $S'$ such that  $RI = RJ$ and $rRr^{-1}I = rRr^{-1}J$, then $I = J$.

\item[(ii)] If $I$ is a right ideal of $S'$. Then $I = IR \cap IrRr^{-1}$. In particular, if $I$ and $J$
are right ideals of $S'$ such that  $IR = JR$ and $IrRr^{-1} = JrRr^{-1}$, then $I = J$.

\end{itemize}
\end{Lemma}

\begin{Proof} $(i)$
Let $I$ be a left ideal of $S'$. Then $RI$ is a left ideal of $R$ and 
$rRr^{-1}I$ is a left ideal of $rRr^{-1}$. Obviously $RI \cap rRr^{-1}I$
contains $I$. 

Consider an element $t \in RI \cap rRr^{-1}I$. By Lemma~\ref{s'principal}, there exist $i_1, i_2 \in I, u \in R, v \in rRr^{-1}$ such that 
$t = ui_1 = vi_2$. Apply Lemma~\ref{s'principal} again to see that there are $t_1,t_2 \in S'$
and $i \in I$ such that $i_1 = t_1i$ and $i_2= t_2i$. Altogether $t = u'i = v'i$, where 
$u' = ut_1 \in R$ and $v' = vt_2 \in rRr^{-1}$. Since everything happens in a division ring, 
$u' = v' \in S'$ and then $t = u'i \in I$. 

A symmetric argument shows $(ii)$.
\end{Proof}

\subsection{Traces of \texorpdfstring{$\End _R(R/rR)$}.} Now we determine $\mathcal{T} (S)$ for 
$S=\End _R(R/rR)$ where  $R$ is a chain domain and $0\neq r\in J(R)$. Hence, $S=S'/rR$ where $S'$ is the ring we have studied in \S \ref{idealizer}; moreover $I=I'/rR$ and $K=K'/rR$.

\begin{Remark}\label{ciclickernels} Since $R_R$ is uniserial, if $t\in S'$ induces a nonzero homomorphism $f\colon R/rR\to R/rR$ then $tR\supsetneq rR$ and the exact sequence (\ref{snakelemma}) becomes

\begin{equation} \label{snakelemma2}0\to \mathrm{Ker}\, f \to R/t'R \stackrel{0}{\to} R/tR \to \mathrm{CoKer}\, f \to 0\end{equation}

Thus $\mathrm{Ker}\, f\cong R/t'R$ is always principal and $R/tR \cong \mathrm{CoKer}\, f $. So that $f\in I\setminus K$ if and only if $t\in R^*\cap rJ(R)r^{-1}$ and $f\in K\setminus I$ if and only if $t\in J(R)\cap rR^*r^{-1}$.
\end{Remark}

Now we can give the following straightened version of Theorem~\ref{TL} for the case of finitely presented modules over chain domains which proves the final part of the Main Theorem.

\begin{Th} \label{TLfp} Let $R$ be a chain domain, let $0\neq r\in J(R)$. Let $U=R/rR$ and let $S=\mathrm{End}_R(U)$.  Assume that there exist $g\in I\setminus K$ and $f\in K\setminus I$ such that $gf= 0$. Then:
\begin{enumerate}
    \item[(1)] $T=SfS=Sf'$ for any monomorphism $f'\in T$ and is pure as a right ideal; 
    \item[(2)] $L=SgS=g'S$ for any epimorphism $g'\in L$, and it is pure as a left ideal.
    \item[(3)] $\mathcal{T}_r(S)=\{0, S, T\}$ and $\mathcal{T}_\ell(S)=\{0, S, L\}$.
\end{enumerate}   
\end{Th}

\begin{Proof}

$(1)$. In view of Theorem~\ref{TL} it remains to see that $T=Sf'$ for any monomorphism $f'\in T$. Following Remark~\ref{ciclickernels}, there is  $t \in J(R)\cap rR^*r^{-1}$
such that $f'$ is induced by multiplication by $t\in S' \setminus I'$. Note that $tr=rt'$ with $t'\in R^*$.

By Lemma~\ref{auxTL}, there exists an epimorphism $g'\in L$ such that $g'f'=0$. Following Remark~\ref{ciclickernels}, there is  $s \in R^*\cap rJ(R)r^{-1}$
such that $g'$ is induced by multiplication by $s\in S' \setminus K'$. Note that $sr=rs'$ with $s'\in J(R)$. By hypothesis, $st \in rR$.

Let $\tau\colon S'\to {\rm End}_R(U)$ be the canonical onto ring homomorphism with kernel $rR$.
Note that if $S'tS' + rR = S't + rR$ then
$\tau(S')\tau(t)\tau(S') = S\tau(t)$ is an idempotent ideal 
of ${\rm End}_R(U)$ finitely generated on the left, hence a trace of a countably 
generated projective right module over $S={\rm End}_R(U)$.
Hence, it is enough to show $S'tS' + rR = S't + rR$.

Note that $R(S'tS' + rR) = RrR$ since $st \in rR$ and $s \in R\setminus J(R)$ imply 
$t \in RrR$. Similarly $R(S't + rR) = RrR$.

Also $rRr^{-1}(S'tS' + rR) = rRr^{-1}tS'+rR$. Now $tr = rt'$ (so $r^{-1} t=t'r^{-1}$) for  $t' \in R\setminus J(R)$
gives $Rr^{-1}t = Rr^{-1}$, so $rRr^{-1}(S'tS'+rR) = rRr^{-1}S'+rR = rRr^{-1}$.
Similarly $rRr^{-1}(S't + rR) = rRr^{-1}t + rR = rRr^{-1}$.

By Lemma~\ref{equalityideals}, we can conclude that $S'tS' + rR = S't + rR$.

$(2)$. The proof that $L=g'S$ for any epimorphism $g'\in L$, is just dual to the one above.
\end{Proof}

\begin{Ex} \label{nearlysimple} Assume that $R$ is a nearly simple chain domain, that is,
a chain domain having exactly 3 two-sided ideals: $0,J(R),$ and $R$. Note that this implies that $0\neq J(R)=J(R)^2$ is not a finitely generated ideal. 

Let $0\neq r,s\in J(R)$,  $R/rR\cong R/sR$ if and only if $RrR=RsR$ by \cite[Proposition~2.21]{puninskibook}. Therefore,  there are only two isomorphism classes of indecomposable finitely presented right modules: $R$ and $R/rR$ for any $0\neq r \in J(R)$ (see \cite[Proposition~14.17]{puninskibook}).

Fix  $0\neq r\in J(R)$. Since $J(R)$ is not finitely generated, there exists $s\in J(R)$ such that $rR\subsetneq sR.$ Therefore, we have the exact sequence
\[0\to sR/rR\stackrel{\varepsilon}\to R/rR \stackrel{\pi}\to R/sR\to 0\]
where $\varepsilon$ denotes the   inclusion and $\pi$ is the canonical projection. Fix   isomorphisms $f'\colon R/rR \to sR/rR$ and $g'\colon R/sR \to R/rR$, and set $f=\varepsilon \circ f'$ and $g =g'\circ \pi$. Then $g\circ f=0$ and $f\in I\setminus K$ while $g\in K\setminus I$. This implies that $U=R/rR$ and $S=\mathrm{End}_R(U)$ fulfill the hypotheses of Theorems~\ref{TL} and \ref{TLfp}. In particular, $S$ has a countably generated indecomposable projective right $S$-module $P$ such that $P/PI\cong S/I$. Note that $P$ is constructed as a direct limit of the sequence
\[S\stackrel{f\times -}{\to}S  \stackrel{f \times-}{\to}S \stackrel{f\times -}{\to} \cdots\]
By Proposition~\ref{equivalencia}, this projective module corresponds to a module $V = P \otimes_{S}U$ which is a direct limit of the sequence
\[U\stackrel{f}{\to}U  \stackrel{f}{\to}U \stackrel{f}{\to} \cdots\]
Since $f$ is mono and not epi, $V$ is a uniserial module which is not finitely generated. In particular, $V$ is isomorphic to a direct summand 
of $U^{(\omega)}$ but not to a direct summand of $U^k$ for any $k \in \N$, that is, $V$ is not quasi-small in the sense of \cite{Dung-Facchini}.

By Example~\ref{computationueum},  $U_e =RrR/rR= J(R)/rR$. Which implies that, in  $S=\End _R (R/rR)$ we have
$T=K$.   

By Example~\ref{computationueum} and because any non-zero cyclic submodule of $R/rR$ is isomorphic to $R/rR$, $U_m=0$. So that, in  $S=\End _R (R/rR)$, we have $L=I$.
\end{Ex}

In view of Example~\ref{nearlysimple} we get the following Corollary of Theorems~\ref{TL} and \ref{TLfp} which gathers some of the main results obtained by Puninski,  and by Dubrovin and Puninski on the endomorphism ring of $U=R/rR$ with $0\neq r\in J(R)$ and $R$ an nearly simple chain domain.

\begin{Cor} \label{cor:nearlysimple} Let $R$ be a nearly simple chain domain, and let $0\neq r\in J(R)$. Set $U=R/rR$ and let  $S =\mathrm{End}_R(U)$. Then
\[I=\{f\in S\mid f\mbox{ is not a monomorphism}\}\]
and 
\[K=\{f\in S\mid f\mbox{ is not an epimorphism}\}.\]
are two uncomparable maximal (right or left) ideals of $S$. 
\begin{enumerate}
\item  For any $f\in K\setminus I$ there exists $g\in I\setminus K$ such that $gf=0$.
\item  For any $g\in I\setminus K$  there exists $f\in K\setminus I$  such that $gf=0$.
\item $K$ is the trace of a countably generated projective and pure right ideal $P$ of $S$. Moreover, $K$ is pure as a right ideal and  $K=Sf$ for any $f\in K\setminus I$.
    \item Any projective right $S$-module is isomorphic to a direct sum of a free right $S$-module  and a direct sum of copies of $P$; therefore, the only ideals of $S$ that are traces of projective right $S$-modules are $0, S$ and $K$. 
    \item $I$  is the trace of a countably generated projective and pure left ideal $Q$ of $S$. Moreover, $I$ is pure as a left ideal and $I=gS$ for any $g\in I\setminus K$;
    \item Any projective left $S$-module is isomorphic to a direct sum of a free left $S$-module and a direct sum copies of $Q$; therefore, the only ideals of $S$ that are traces of projective left $S$ modules are $0, S$ and $I$.
\end{enumerate}
\end{Cor} 

The reader may also want to have a look to \cite[Theorem~11.87]{libro2} for a nice and clear account on the amazing properties of the ring $S$.

\section{Non-serial direct summand of a serial module via lifting modulo the trace ideal of a projective module}

Recall that a module is serial if it is a direct sum of uniserial modules.  In this section we will show how to obtain
Puninski's example   of  a direct summand of a serial module which 
is not a direct sum of indecomposable modules \cite{Pun01} by using the lifting of projective modules modulo a trace ideal.  We briefly recall this topic in the next subsection.

\subsection{Lifting projective modules modulo a trace ideal.}

\begin{Lemma}
\label{traces.quo}
\cite[Theorem 3.1, Corollary 3.2]{traces} \cite[Proposition 3.4]{wiegand}
Let $\Lambda$ be a ring, and let $T$ be the trace of a projective right $\Lambda$-module. Let $P'$ be a projective right $\Lambda/T$-module. Then,
\begin{enumerate}
    \item[(1)] There is a projective right $\Lambda$-module $X$ such that $T\subseteq \Tr_\Lambda(X)$, $X/XT\cong P'$ and $\Tr_\Lambda(X)/T=\Tr_{\Lambda/T}(P')$.
    \item[(2)] If $T \in \mathcal T(\Lambda)$ and $P'$ is countably generated, then $X$ can be taken to be countably generated.
    \item[(3)] Let $T \in \mathcal T(\Lambda)$, let $P$ be a countably generated projective right $\Lambda$-module with trace ideal $T$, and assume that $P'$ is finitely generated. Let $X_1$ and $X_2$ be countably generated projective right $\Lambda$-modules such that, for $i=1,2$, $X_i/X_iT\cong P'$. Then $X_1\oplus P^{(\omega)}\cong X_2\oplus P^{(\omega)}$.
\end{enumerate}
\end{Lemma}

We are going to apply Lemma~\ref{traces.quo} in the context described in the following remark.

\begin{Remark} \label{mplusM}
Let $M$ and $M'$ be two right modules over a ring $R$, and with endomorphism rings $S$ and $S'$, respectively.
Let	
	$$\Lambda:=\End_R(M'\oplus M)=\begin{pmatrix} S'&\Hom_R (M,M')\\ \Hom_R (M',M)&S \end{pmatrix}$$

Consider the finitely generated projective right $\Lambda$-module $$P=\left(\begin{smallmatrix}1&0\\ 0&0\end{smallmatrix}\right)\Lambda\cong \Hom_R(M'\oplus M,M').$$  Its trace ideal is 
    $$T=\Lambda\begin{pmatrix}1&0\\ 0&0\end{pmatrix}\Lambda=\begin{pmatrix}S'&\Hom_R (M,M')\\ \Hom_R (M',M)&\Hom_R (M',M)\Hom_R (M,M')\end{pmatrix}.$$ 
Then $\Lambda/T\cong S/\Hom_R (M',M) \Hom_R (M,M')$. Notice that 
\[\Hom_R (M',M) \Hom_R (M,M')=\{f\in S\mid f \mbox{ factors through $(M')^n $ for some $n$}\}.\]

Therefore, by Lemma~\ref{traces.quo}, if 
$S/\Hom_R (M',M) \Hom_R (M,M')\cong Q'\oplus P'$ then there exists countably generated projective right $\Lambda$-modules $X\cong X \oplus P^{(\omega)}$ and $Y\cong Y \oplus P^{(\omega)}$ such that $X/XT\cong P'$, $Y/YT\cong Q'$ and $X\oplus Y\cong P^{(\omega)}\oplus S$.

By Proposition~\ref{equivalencia}, if $M'$ and $M$ are finitely generated and we set $A=X\otimes _\Lambda (M'\oplus M) $ and $B=Y\otimes _\Lambda (M'\oplus M)$ we get a decomposition:
\[(M')^{(\omega)}\oplus M\cong A\oplus B.\]

Set $M'=R$. Then $\Lambda/T\cong S/M  \Hom_R (M,R)$ is the endomorphism ring of $M_R$ in the stable category. Let $e=\left(\begin{smallmatrix}1&0\\ 0&0\end{smallmatrix}\right)$. Then, $\Lambda$, $e\Lambda e$, $P_\Lambda$, and $\Hom _\Lambda (P,\Lambda)=\Lambda e$ form a Morita context named  as \emph{Auslander context} in \cite[Example~2.4]{buchweitz}.
\end{Remark}

\subsection{Nearly simple, coherent chain rings that are not domains.}

By  the Drozd-Warfield Theorem \cite[Theorem~3.29, Corollary~3.30]{libro}, if $R$ is a serial ring  every finitely presented $R$-module is 
a direct sum of modules of the form $R/rR$ where $r \in J(R)$. Moreover, by \cite[Proposition~2.21, Proposition~2.18]{puninskibook},
$R/rR \cong R/sR$ if and only if $RrR = RsR$ for any $r,s \in R$ if and only if there are $u,v \in R \setminus J(R)$
such that $r = usv$.

Let $R$ be a nearly simple chain ring (that is, $R$ is a chain ring with exactly three two-sided ideals $R$, $J(R)$ and $0$, and $J(R)^2\neq 0$) which is not a domain.  We also assume that $R$ is  coherent, that is, every 
finitely generated right or left ideal of $R$ is finitely presented. 

 Since we assume that there are $0 \neq x,y \in J(R)$ such that 
$xy = 0$, we immediately see from the remarks above that every element of $J(R)$ has non-trivial left and right annihilator. 
In particular, if $0 \neq j \in J(R)$, then $jR$ is an indecomposable 
finitely presented module which is not projective, that is, $jR \cong R/rR$
for some (any) $0 \neq r \in J(R)$.
Also note that if $f \in {\rm End}_R(R/rR)$ then ${\rm Im}\ f$ is finitely presented 
and thus ${\rm Ker}\ f$ has to be finitely generated. 

For any subset $X$ of $R$ let us denote by $\mathbf{l. ann}_R(X)$ (by $\mathbf{r. ann}_R(X)$) its left (right) annihilator in $R$. The above remarks are the key to prove that if $r\in R$ then $$\mathbf{r. ann}_R(\mathbf{l. ann}_R(r))=rR,$$ by the symmetry of the hypotheses also $$\mathbf{l. ann}_R(\mathbf{r. ann}_R(r))=Rr,$$ see \cite[Lemma~15.2]{puninskibook} for further details. This is to say, that if $R$ is a coherent nearly simple chain ring  that is not a domain then $R$ satisfies the double annihilator condition for finitely generated right and left ideals. In the next Lemma we observe that, using a classical result of Ikeda and Nakayama, this implies that    the Baer criteria holds for finitely generated ideals of $R$.

\begin{Lemma} \label{ikeda-nakayama} Let $R$ be  a coherent, nearly simple chain ring which is not a domain. Then, for any finitely generated (equivalently, principal) right or left ideal $I$ of $R$, any module morphism $I\to R$ is given by multiplication by an element of $R$.
\end{Lemma}

\begin{Proof} A celebrated lemma by Ikeda and Nakayama, see \cite[Theorem~12.4.2]{kasch}, states that if $R$ is any ring that satisfies:
\begin{itemize}
    \item[(a)] For any pair $I$ and $J$ of right ideals  of $R$, $\mathbf{l. ann}_R(I\cap J)= \mathbf{l. ann}_R(I)+\mathbf{l. ann}_R( J)$,
    \item[(b)] For any finitely generated left ideal $K$ of $R$ then $\mathbf{l. ann}_R(\mathbf{r. ann}_R(K))=K$,
\end{itemize}
then any morphism from a finitely generated right ideal of $R$ to $R$ is given by left multiplication by an element of $R$. 

If $R$ is a chain ring then condition $(a)$ is trivially satisfied. By the remarks before the Lemma, our hypotheses imply that $(b)$ is also satisfied. So we can conclude the claim for finitely generated right ideals and, by symmetry, also for finitely generated left ideals.     
\end{Proof}

\begin{Rem} \label{manymorphisms} Let $R$ be  a coherent, nearly simple chain ring which is not a domain, and let $0\neq r\in J(R)$. Our aim is to determine the endomorphisms of $U=R/rR$ that factor through a finitely generated free module.

Since $\mathbf{l. ann}_R(r)=Rs$ and $\mathbf{r. ann}_R(s)=rR$, we deduce that left multiplication by $s$ induces an isomorphism $\varphi \colon U=R/rR\to  sR$. Therefore, if   $f\colon R/rR\to R/rR$ is a homomorphism then, by Lemma~\ref{ikeda-nakayama}, $g=\varphi\circ f \circ \varphi^{-1} \colon sR\to sR$ is given by left multiplication by $t\in R$. Therefore, we have a commutative diagram 
\[\begin{tikzcd}%[row sep=tiny,column sep=tiny]
        sR \rar{g}  \drar{\varepsilon} & sR \rar{\varepsilon} & R  \\
        & R \urar{t}
    \end{tikzcd}\]
where $\varepsilon$ denotes the inclusion.

If $t\in sR$ then we deduce that $g$ and, hence, $f$ factor through $R$. Otherwise, $sR \subsetneq tR$. In this case, if $f$ is not injective then $t\in J(R)$ and, by the remarks above, $tR\cong R/rR$, so we can deduce that $h\circ f$ factors through $R$ for a suitable monomorphism $h\in \mathrm{End}_R (R/rR)$. 

That is, we are showing that, provided  $f$ is not injective,  there exists a nonomorphism $h$ such that $h\circ f$ factors through $R$. We want to prove that if, in addition, $f$ is not onto then $f$ factors through $R$.
\end{Rem}

We need the following lemma which is similar to \cite[Lemma 7.2]{PP16}. 

\begin{Lemma} \label{PP16}
    Let $R$ be a chain ring, $0 \neq r \in R$, and let $U = R/rR$ be a uniserial module with endomorphism ring $S=\mathrm{End}_R(U)$.  If $f,g \in S$ have equal kernels then there exists a monomorphism 
    $m \in S$ such that either $mf = g$ or $mg = f$. Moreover, if ${\rm Im}\ f$,${\rm Im}\ g \varsubsetneq U_e$, $m$ can be chosen to be an automorphism. 
\end{Lemma}

\begin{proof} Let $S'=\{a\in R\mid ar\in rR\}$. As remarked in \S~\ref{idealizer}, any element of $S$ is given by left multiplication by an element of $S'$ and, hence, $S\cong S'/rR$.

We may assume that $f$ and $g$ are nonzero. Assume that left multiplication of $x \in S'$ induces endomorphism $f$,
i.e., $f(u+rR) = xu +rR$ for any $u \in R$ and that left multiplication of $y\in S'$ induces endomorphism $g$.
Since $f,g \neq 0$, there exist $s,s' \in J(R)$ such that $xs = r$ and $ys' = r$. Obviously, ${\rm Ker}\ f = 
sR/rR$ and ${\rm Ker}\ g = s'R/rR$. Since ${\rm Ker}\ f = {\rm Ker}\ g$, $sR = s'R$, i.e., $sv = s'$ for some 
$v \in R^{*}$.

Since $R$ is a chain ring, $Rx \subseteq Ry$ or $Ry \subseteq Rx$ holds. Assume the former. 
Let $z \in R$ be such that $x= zy$. We claim that $z\in S'$: Suppose that it is not the case. 
Then there exists $j \in J(R)$ such that $zrj = r$. Then $zys'j = r$ and also $r = xs'j = xsvj = rvj$.
Then $r(1-vj) = 0$ and, since $(1-vj)$ is invertible, $r = 0$. It is not the case, hence $zr \in rR$.

Since $z\in S'$, left multiplication by $z$ induces $m \in S$ and $x = zy$ implies $f = mg$.
Similarly, if $Ry \subseteq Rx$ there exists $m \in S$ such that $g = mf$. Since we assume 
${\rm Ker}\ f = {\rm Ker}\ g$, $m$ has to be a monomorphism. 

Assume now that $g = mf$, where  $m$ is not bijective and that $f(U)\varsubsetneq U_e= \sum_{h \in S \setminus K} {\rm Ker}\, h$. 
Then, there exists $h_0\in S \setminus K$ such that $f(U)\subseteq \mathrm{Ker}\, h_0$. Then $(m+h_0)f=f$ and $m+h_0$ is invertible by Remark~\ref{monoepi}.

Similarly if $f = mg$, where $m$ is not bijective and ${\rm Im}\ g \subseteq U_e$ we find an automorphism $m'$ such that $f = m'g$.
\end{proof}

\begin{Prop} \label{stable}
Let $R$ be  a coherent, nearly simple chain ring which is not a domain. Let $U$ be the indecomposable finitely presented uniserial $R$-module which is not free and let $S = {\rm End}_R(U)$. 
Then  \[J(S)=\{f\in S\mid f \mbox{ factors through a finitely generated free module}\}.\]
\end{Prop}
\begin{Proof}
Fix some $0 \neq r \in J(R)$. Then we may assume that $U = R/rR$.

Let $f\in S$ be such that it factors through a free module. If $f$ is injective then, as $U$ is uniserial,   there is an injective morphism $f'\colon R \to U$ that cannot be surjective. Let $\pi\colon R\to U$ be the canonical projection, then $f'+\pi$ must be an isomorphism (cf. Remark~\ref{monoepi}) which is not possible.

Assume now that $f$ is onto and that it factors through a finitely generated free module. But $\mathrm{Hom}_R(U, R)=\mathrm{Hom}_R(U, J(R))$, so that $U=f(U)=UJ(R)$ which is not possible because $U$ is finitely generated.

This proves that if $f$ factors through a free module, then $f\in I\cap K=J(S)$.

Let $0 \neq f \in J(S)$. We want to show that $f$ factors through a finitely generated free module. By Example~\ref{computationueum}, $U_e=RrR/rR = J(R)/rR$, so ${\rm Im}\ g$
is strictly contained in $U_e$  for any $g \in K$. By Remark~\ref{manymorphisms}, there exists a monomorphism   $h\in S$ such that $h\circ f$ factors through $R$. Since  ${\rm Ker}\ h\circ f = {\rm Ker}\ f$, by Lemma~\ref{PP16} there exists an isomorphism $m \in S$ such that 
$f = mhf$. Then $f$ factors through a free  module, as we wanted to prove.
\end{Proof}

\begin{Ex}\label{fgcoherentns} Let $R$ be a coherent, nearly simple chain ring which is not a domain. Let $0\neq r\in J(R)$. Let $U=R/rR$, and let $S=\mathrm{End}_R(U)$.

By the the remarks preceding Lemma~\ref{ikeda-nakayama}, $\mathbf{l.ann} (r)=Rs$, so that $R/rR\cong sR$. Since $\mathbf{l.ann} (s)$ is principal, and $J(R)$ is not finitely generated there exists $t\in J(R)$ such that $ts\neq 0$ and, hence, $0\neq t'R=\mathbf{r.ann} (t) \varsubsetneq sR$. Therefore, there is an exact sequence
\[0\to t'R \to sR \stackrel t \to tsR\to 0.\]
 Since $t'R\cong sR\cong tsR\cong rR$, we deduce that the have an exact sequence 
 \[0\to R/rR \stackrel f  \to R/rR \stackrel g  \to R/rR \to 0. \]
 As in Example~\ref{nearlysimple} and because $f$ is mono and not epi, we deduce that if $V$ denotes the direct limit of the direct system 
 \[U\stackrel{f}{\to}U  \stackrel{f}{\to}U \stackrel{f}{\to} \cdots\]
 then $V$ is a uniserial module which is not finitely generated that is isomorphic to a direct summand 
of $U^{(\omega)}$. 
 \end{Ex}

\begin{Lemma} \label{lambdaradicalfg} Let $R$ be a coherent, nearly simple chain ring which is not a domain. Let $0\neq r\in J(R)$ and set $U=R/rR$.  Let $S=\mathrm{End}_R(U)$.

Let $\Lambda =\mathrm{End}_R(R\oplus U)$ then,
\begin{itemize}
    \item[(i)] $\Lambda \cong \begin{pmatrix}
        R&\ell _R(r) \\ U&S
    \end{pmatrix}$
where $\ell _R(r)$ denotes the left annihilator of $r$ in $R$.
\item[(ii)] The isomorphism in $(i)$ restricts to an isomorphism $J(\Lambda)\cong \begin{pmatrix}
        J(R)&\ell _R(r) \\ U&J(S)
    \end{pmatrix}$.
\item[(iii)] There is a countably generated projective right $\Lambda$-module $Y$ such that $$Y/YJ(\Lambda)\cong S/I.$$ 
\item[(iv)] Let $T=\Lambda \begin{pmatrix}
    1&0\\ 0&0
\end{pmatrix}\Lambda$ then $Y/YT\cong S/I$.

\item[(v)] Following the  notation of Example~\ref{fgcoherentns}, $Y\otimes _\Lambda (R\oplus U)\cong V$.
\end{itemize}

Moreover, the category of pure projective right modules over $R$ is equivalent to the category of projective right $\Lambda$-modules.
\end{Lemma}

\begin{Proof} $(i)$. As
$\Lambda \cong  \begin{pmatrix} \mathrm{End}_R(R) &\mathrm{Hom}_R(R/rR, R)\\ \mathrm{Hom}_R(R, R/rR) &\mathrm{End}_R(R/rR)
\end{pmatrix},$
the isomorphisms $\mathrm{Hom}_R(R/rR, R)\cong \ell _R(r)$ and $\mathrm{Hom}_R(R, R/rR)\cong R/rR$ yield the isomorphism claimed in $(i)$.

$(ii)$. If $s\in \ell _R(r)$ then $$\begin{pmatrix}
        0&s \\ 0&0
    \end{pmatrix}\begin{pmatrix}
        R&\ell _R(r) \\ U&S
    \end{pmatrix}=\begin{pmatrix}
        sR&sS \\ 0&0
    \end{pmatrix}$$
    which is easily seen to be contained in $J(\Lambda)$ because $sR\le J(R)$.

    If $t+rR \in U$ then $$\begin{pmatrix}
        0&0 \\ t+rR&0
    \end{pmatrix}\begin{pmatrix}
        R&\ell _R(r) \\ U&S
    \end{pmatrix}=\begin{pmatrix}
        0&0 \\ (t+rR)R&(t+rR)\ell _R(r)
    \end{pmatrix}$$ which is contained in $J(\Lambda)$ because by Proposition~\ref{stable}, $(t+rR)\ell _R(r)\in J(S)$. Now it follows that $J(\Lambda)$ is as claimed in $(ii)$.

$(iii)$. By Example~\ref{fgcoherentns}, there exists $f$ and $g\in S$ such that $f\in K\setminus I$ and $g\in I\setminus K$, $gf=0$ and $f+g$ is invertible. Therefore
$F= \begin{pmatrix}
        0&0 \\ 0&f
    \end{pmatrix}$
and $G=\begin{pmatrix}
        1&0 \\ 0&g
    \end{pmatrix}$ are in $\Lambda$, and are such that $GF=0$ and $F+G$ is invertible. Since $F=F\begin{pmatrix}
        0&0 \\ 0&1
    \end{pmatrix}$ the morphism of $\Lambda$ given by left multiplication by $F$ induces an  endomorphism 
    \[\begin{pmatrix}
        0&0 \\ R/rR&S
    \end{pmatrix}\stackrel{F\times} \longrightarrow  \begin{pmatrix}
        0&0 \\ R/rR&S
    \end{pmatrix}.\]
Let $X=\begin{pmatrix}
        0&0 \\ R/rR&S
    \end{pmatrix}$. By the remarks before Proposition~\ref{rightleft}, the direct limit of the countable direct system 
    \[X\stackrel{F\times} \longrightarrow X \cdots X\stackrel{F\times} \longrightarrow X \cdots \qquad (*)\]
    is a countably generated projective right $\Lambda$-module $Y$. Tensoring $(*)$ by $\Lambda /J(\Lambda)$ we deduce that $Y/YJ(\Lambda)\cong S/I$.

$(iv)$. Note that $XT=XJ(\Lambda)$. Therefore the claim follows from $(iii)$.

$(v)$. Let $e_1=\begin{pmatrix}
    1&0\\ 0&0
\end{pmatrix}$ and $e_2=\begin{pmatrix}
    0&0\\ 0&1
\end{pmatrix}$. Notice that \begin{align*}
    R\oplus R/rR\cong & \Lambda \otimes _\Lambda 
\left(R\oplus R/rR\right)=(e_1\Lambda \otimes _\Lambda \left(R\oplus R/rR\right)) \oplus (e_2\Lambda \otimes _\Lambda \left(R\oplus R/rR\right))=\\ &=\left( \begin{pmatrix}
        R&\ell _R(r) \\ 0&0
    \end{pmatrix}\otimes_\Lambda (R\oplus R/rR)\right)\oplus \left( \begin{pmatrix}
        0&0 \\ R/rR&S
    \end{pmatrix}\otimes _\Lambda (R\oplus R/rR)\right)
    \end{align*}
Since $e_2 (R\oplus R/rR)=R/rR$, applying the functor $-\otimes _\Lambda \left(R\oplus R/rR\right)$ to the direct system $(*)$ yields the direct system 
\[U\stackrel{f}{\to}U  \stackrel{f}{\to}U \stackrel{f}{\to} \cdots\]
whose limit is the $R$-module $V$ of Example~\ref{fgcoherentns}. So that, $Y\otimes _\Lambda (R\oplus U)\cong V$ as we wanted to show.

The final part of the statement is a consequence of the Drodz-Warfield Theorem and Proposition~\ref{equivalencia}, because it follows from these results that the category of pure-projective right $R$ modules is $\mathrm{Add}\, (R\oplus U)$ which is equivalent to $\mathrm{Add}\, \Lambda$.
\end{Proof}

\begin{Th}[Proposition~15.15 in \cite{puninskibook}] \label{coherent}
Let $R$ be a coherent, nearly simple chain ring which is not a domain. Let $0\neq r\in J(R)$. Then 
\[R^{(\N)}\oplus R/rR\cong (R^{(\N)}\oplus V)\oplus W\]
where $V$ is the non-quasi small module constructed in Example~\ref{fgcoherentns}, and $W$
is a pure projective $R$-module which is not a direct sum 
of indecomposable modules. In particular, $W$ is a direct summand of a serial module that is not serial.
\end{Th}

\begin{Proof}
Let $U = R/rR$ for some fixed $0\neq  r \in J(R)$. Let $\Lambda = {\rm End}_R(R \oplus U)$. By Lemma \ref{lambdaradicalfg}, the category of pure projective $R$-modules is 
${\rm Add}(R \oplus U)$, which is equivalent to ${\rm Proj}$-$\Lambda$. 
Note that indecomposable objects in ${\rm Add}(R\oplus U)$ or ${\rm Proj}$-$\Lambda$ are those with trivial idempotents 
in their endomorphism rings - so indecomposable modules of ${\rm Add}(R\oplus U)$ correspond to indecomposable projective $\Lambda$-modules. 

Now we follow the construction of Remark~\ref{mplusM} with $M'=R$ and $M=U=R/rR$. Let $S=\mathrm{End}_R(U)$. By Proposition~\ref{stable}, the endomorphism of $U$ in the stable category is $S/J(S)\cong S/I\times S/K$.

Set $P'_\Lambda=\mathrm{Hom}_R(R\oplus U, R)\cong \begin{pmatrix}
    1&0\\ 0&0
\end{pmatrix}\Lambda $ and let $T$ be its trace in $\Lambda$. Then there exist countably generated projective right $\Lambda$-modules $P\cong P\oplus (P')^{(\omega)}$ and $Q\cong Q\oplus (P')^{(\omega)}$ such that, as right $S$-modules, $P/PT\cong S/I$ and $Q/QT\cong S/K$; moreover, $(P\oplus Q) \otimes _\Lambda (R\oplus R/rR)\cong R^{(\omega)}\oplus U$.

By Lemma~\ref{lambdaradicalfg}, $\Lambda$ has a countably generated projective right module $Y$ such that $Y/YJ(\Lambda) =Y/YT\cong S/I$ and such that $Y\otimes _\Lambda (R\oplus R/rR)\cong V$. By Lemma~\ref{traces.quo}, $P\cong (P')^{(\N)}\oplus Y$. Moreover $P\otimes _\Lambda (R\oplus R/rR)\cong R^{(\N)}\oplus V$.

Now we prove that $Q$ does not decompose as a sum of indecomposable modules. Assume that $Q = \oplus_{i \in I} Q_i$, with $Q_i$ indecomposable. Since $Q/QT$ is simple, $Q_i = Q_iT$ for every $i$ with exactly $1$
exception. Also note that $Q_i = Q_iT$ means that $Q_i \in {\rm Add}(P')$ and since every projective $R$-module is free,  
$Q_i \cong P'$. Thus we may write $Q = Q_0 \oplus P'^{(\omega)}$ where $Q_0$ is an indecomposable module such that $Q_0/Q_0T \cong S/K$.
To conclude that there is no such $Q_0$ we can follow \cite{pavel}. Note that $\Lambda$ is a semilocal 
ring and $\Lambda/J(\Lambda) \cong R/J(R)\times S/J(S) \cong R/J(R) \times S/I \times S/K$. Let $X_1,X_2,X_3$ be simple modules 
related to components $R/J(R),S/I,S/K$. Recall that every projective $\Lambda$-module $M$ is determined up to isomorphism by its factor $M/MJ(\Lambda)$.
By our construction $Q/QJ(\Lambda) \cong X_1^{(\omega)} \oplus X_3$. If $Q_0/Q_0J(\Lambda)$ contains a direct summand isomorphic to $X_1$, then 
$P_1$ is a direct summand of $Q_0$, so the only possibility is that $Q_0/Q_0J(\Lambda) = X_3$. Also we know there exists a projective $\Lambda$-module 
$Y$ such that $Y/YJ(\Lambda) \cong X_2$. But then $T \oplus Q_0 \cong \mathrm{Hom}_R(R\oplus U, U)$, which is indecomposable. This contradiction proves that such indecomposable module $Q_0$ does not exist. Therefore, $Q$ is not a direct sum of indecomposable submodules as claimed.  

By Proposition~\ref{equivalencia}, $W=Q\otimes _\Lambda (R\oplus R/rR)$ does not decompose as a sum of indecomposable submodules. In particular, it is not serial. 

To conclude the proof, use Proposition~\ref{equivalencia}, to establish the isomorphisms, \[R^{(\N)}\oplus R/rR\cong (P\oplus Q)\otimes _\Lambda (R\oplus R/rR)\cong  (R^{(\N)}\oplus V)\oplus W\]  
claimed in the statement.
\end{Proof}

\begin{Remark}
  Keeping the notation from Theorem \ref{coherent} it is possible to prove that if $R$ is a coherent nearly simple chain ring which is not a domain then any pure projecive right $R$-module is 
  a direct sum $\oplus_{i \in I} U_i$ where each $U_i$ is isomorphic to a module from $\{R,R/rR,V,W\}$. By the final statement in Lemma~\ref{lambdaradicalfg}, this is equivalent to say that any projective right $\Lambda$-module is isomorphic to a direct sum of 
  modules from the set $\{e_1\Lambda, e_2\Lambda,Q,Y\}$ (cf \cite[Theorem~5.1]{pavel}). A similar result can be proved  for projective left $\Lambda$-modules. 
  Let us stress that the key point is the symmetry between ${\rm Proj}$-$\Lambda$ and $\Lambda$-${\rm Proj}$ - each module from the set $\{e_1\Lambda, e_2\Lambda,Q,Y\}$ has its companion in $\Lambda$-${\rm Proj}$
  and hence arguments from \cite[Section~5]{pavel} apply also on the left.
  
  Let us denote  by $P_1 = \Lambda e_1,P_2 = \Lambda e_2$ the indecomposable direct summands of ${}_{\Lambda}\Lambda$. If $f \in K\setminus I, g \in I \setminus K$ satisfy $gf = 0$, we denote 
  $$F' = \left(\begin{array}{cc} 1 & 0 \\0 & f \end{array}\right), G' = \left(\begin{array}{cc} 0 & 0 \\0 & g \end{array}\right) $$
  The equality $G'(G'+F') = G'^2$ and $G'+F' \in \Lambda^{*}$
  give a countably generated projective left $\Lambda$-module $Y'$ which is a direct limit of the sequence 
  $${}_{\Lambda}\Lambda \stackrel{- \times G'}{\to} {}_{\Lambda}\Lambda \stackrel{- \times G'}{\to} {}_{\Lambda}\Lambda \stackrel{- \times G'}{\to} \cdots $$
  and such that $Y'/J(\Lambda)Y'\cong S/K$. Finally, the trace ideal $T = {\rm Tr}_{\Lambda}(P_1)$ and the decomposition $S/T \cong S/I \times S/K$ give a projective left $\Lambda$-module $Q'$ such that $Q' \oplus P_1^{(\omega)} \cong Q'$
  and $Q'/TQ' \cong S/I$. 
  
  By Lemma~\ref{lambdaradicalfg}, $\Lambda/J(\Lambda) \cong R/J(R) \times S/I \times S/K$. Let $X_1',X_2',X_3'$ be simple left $\Lambda$-modules related to these components. 
  For a countably generated projective left $\Lambda$-module $P$ we write ${\rm dim}(P) = (\alpha,\beta,\gamma)$ if $P/J(\Lambda)P \cong {X'_1}^{(\alpha)} \oplus {X'_2}^{(\beta)} \oplus {X'_3}^{(\gamma)}$.
  Direct computation gives:
  \[ {\rm dim}(P_1) = (1,0,0), {\rm dim}(P_2) = (0,1,1),{\rm dim}(Y') = (0,0,1),{\rm dim}(Q') = (\omega,1,0)\,.\]
  If $P$ is a countably generated projective left $\Lambda$-module and ${\rm dim}(P) = (\omega,\beta,\gamma)$, then $P \oplus P_1^{(\omega)} \cong P$ (hence $P$ is $T$-big in the terminology of \cite{AHP2}) and it is easy to see that 
  $P \cong P_1^{(\omega)} \oplus Q'^{(\beta)} \oplus Y'^{(\gamma)}$.
  
  If $P$ is a countably generated projective left $\Lambda$-module and ${\rm dim}(P) = (0,\beta,\gamma)$ then $P_2^{(\omega)} \cong P \oplus P_2^{(\omega)}$, hence $P \in {\rm Add}(P_2)$. The categories ${\rm Add}(P_2)$
  and ${\rm Add}({}_SS)$ are equivalent. In this case we can apply Proposition~\ref{gfzero}(iii) (or repeat the analysis preceding this result) to see that $P$ is isomorphic to a direct sum of modules 
  from the set $\{P_2,Y'\}$. 

  If $P$ is a countably generated projective left $\Lambda$-module and ${\rm dim}(P) = (\alpha,\beta,\gamma)$ for a finite nonzero $\alpha$ then, by Proposition~\ref{trfg}, $P \cong P_1^{\alpha} \oplus P'$
  where ${\rm dim}(P') = (0,\beta',\gamma')$ and hence $P'$ is a direct sum of modules from the set $\{P_2,Y'\}$. 

  Following the notation introduced in Examples~\ref{tracesd1d2}, the trace ideals  of projective left $\Lambda$-modules correspond to the sets $\emptyset, \{1,2,3\},   \{1\} , \{2,3\}, \{3\}, \{1,2\}$ which correspond to the traces of $0$, $\Lambda$, $P_1$, $P_2$, $Y'$ and $Q'$, respectively. While trace ideals of projective right $\Lambda$-modules correspond to the sets $\emptyset, \{1,2,3\},   \{1\} , \{2,3\}, \{2\}, \{1,3\}$.
\end{Remark}

\bibliographystyle{amsplain}%Used BibTeX style is unsrt
\bibliography{references}
\end{document}